\documentclass{amsart}
\usepackage{times,array,amssymb,tikz,xcolor}
\usetikzlibrary{petri}
\usepackage{cite}

\newcommand{\ccirc}{\kern0.5ex\vcenter{\hbox{$\scriptstyle\circ$}}\kern0.5ex}

\newtheorem{Theorem}[subsection]{Theorem}
\newtheorem{Lemma}[subsection]{Lemma}
\newtheorem{Proposition}[subsection]{Proposition}
\newtheorem{Corollary}[subsection]{Corollary}
\newtheorem{theorem}{Theorem}

\theoremstyle{definition}
\newtheorem*{Example}{Example}

\newtheorem{Definition}[subsection]{Definition}
\newtheorem*{Remark}{Remark}
\numberwithin{equation}{section}

\def\H{\mathbb H}
\def\P{\mathbb P}

\def\GL{\mathrm{GL}}

\def\SL{\mathrm{SL}}
\def\SU{\mathrm{SU}}

\def\tr{\operatorname{tr}}

\def\N{\mathbb N}
\def\C{\mathbb C}

\def\P{\mathbb P}
\def\Q{\mathbb Q}
\def\R{\mathbb R}
\def\x{\times}

\def\Z{\mathbb Z}

\def\mod{\  \mathrm{mod}\ }

\def\Im{\mathrm{Im\,}}

\def\SL{\mathrm{SL}}

\def\wt{\widetilde}
\def\M#1#2#3#4{\begin{pmatrix}#1&#2\\#3&#4\end{pmatrix}}
\def\SM#1#2#3#4{\left(\begin{smallmatrix}#1&#2\\#3&#4\end{smallmatrix}
  \right)}

\def\mb #1{\mathbf{#1}}

\def\rm #1{\mathrm{#1}}

\def\X #1,#2{X^{#1}_0(#2)}
\def\braces #1{\left\{#1\right\}}
\def\abs#1{\left|#1\right|}

\begin{document}
%\fontsize{14pt}{14pt}\selectfont

\title[modular differential equations]{Metrics with Positive constant curvature and modular differential equations}

\author{Jia-Wei Guo, Chang-Shou Lin, and Yifan Yang}

\address{Department of Mathematics, National Taiwan
University, Taipei, Taiwan 10617.}
\email{jiaweiguo312@gmail.com}

\address{Center for Advanced Study in Theoretical Sciences (CASTS), National Taiwan University, Taipei, Taiwan 10617.} 
\email{cslin@math.ntu.edu.tw}

\address{Department of Mathematics, National Taiwan
University and National Center for Theoretical Sciences, Taipei,
Taiwan 10617.}
\email{yangyifan@ntu.edu.tw}

\begin{abstract}
  Let $\H^\ast=\H\cup\Q\cup\{\infty\}$, where $\H$ is the complex
  upper half-plane, and $Q(z)$ be a meromorphic modular form of weight
  $4$ on $\SL(2,\Z)$ such that the differential equation
  $\mathcal L:y''(z)=Q(z)y(z)$ is Fuchsian on $\H^\ast$. In this paper, we
  consider the problem when $\mathcal L$ is apparent on $\H$, i.e.,
  the ratio of any two nonzero solutions of $\mathcal L$ is
  single-valued and meromorphic on $\H$. Such a modular differential
  equation is closely related to the existence of a conformal metric
  $ds^2=e^u|dz|^2$ on $\H$ with curvature $1/2$ that is invariant
  under $z\mapsto\gamma\cdot z$ for all $\gamma\in\SL(2,\Z)$.

  Let $\pm\kappa_\infty$ be the local exponents of $\mathcal L$ at
  $\infty$. In the case $\kappa_\infty\in\frac12\Z$, we obtain the
  following results:
  \begin{enumerate}
    \item[(a)] a complete characterization of $Q(z)$ such that
      $\mathcal L$ is apparent on $\H$ with only one singularity (up
      to $\SL(2,\Z)$-equivalence) at $i=\sqrt{-1}$ or
      $\rho=(1+\sqrt3i)/2$, and
    \item[(b)] a complete characterization of $Q(z)$ such that
      $\mathcal L$ is apparent on $\H^\ast$ with singularities only at
      $i$ and $\rho$.
  \end{enumerate}
  We provide two proofs of the results, one using Riemann's existence
  theorem and the other using Eremenko's theorem on the existence of
  conformal metric on the sphere.

  In the case $\kappa_\infty\notin\frac12\Z$, we let
  $r_\infty\in(0,1/2)$ be defined by
  $r_\infty\equiv\pm\kappa_\infty\mod 1$. Assume that
  $r_\infty\notin\{1/12,5/12\}$. A special case of an earlier result
  of Eremenko and Tarasov says that $1/12<r_\infty<5/12$ is the
  necessary and sufficient condition for the existence of the
  invariant metric. The threshold case $r_\infty\in\{1/12,5/12\}$ is
  more delicate. We show that in the threshold case, an invariant
  metric exists if and only if $\mathcal L$ has two linearly
  independent solutions whose squares are meromorphic modular forms of
  weight $-2$ with a pair of conjugate characters on $\SL(2,\Z)$.
  In the non-existence case, our example shows that the monodromy data
  of $\mathcal L$ are related to periods of the elliptic curve
  $y^2=x^3-1728$.
\end{abstract}
\date{\today}
\thanks{We thank the referee for providing the reference
  \cite{Eremenko-Tarasov} and for detailed comments that improve the
  exposition of the paper.}

\maketitle

\section{Introduction}
A meromorphic function $Q$ on the upper half-plane $\H$ is called a
meromorphic modular form of weight $k\in\Z$ with respect to
$\SL(2,\Z)$ if $Q$ satisfies
$$
Q(\gamma\cdot z)=(cz+d)^kQ(z),\quad \gamma=\M abcd\in\SL(2,\Z),
$$
and $Q$ is also meromorphic at the cusp $\infty$. 
When $k=0$, a meromorphic modular form is called a modular function.
We refer to \cite{Apostol} and \cite{Serre} for the elementary theory
of (holomorphic) modular forms. Given a meromorphic modular form $Q$
of weight $4$ on $\SL(2,\Z)$, we consider a Fuchsian modular
differential equation of second order on $\H$
\begin{equation}\label{(1.1)}
y''=Q(z)y\quad \text{on}\ \H,\qquad y':=\frac{dy}{dz}.
\end{equation}
The differential equation \eqref{(1.1)} is called Fuchsian if the
order of any pole of $Q$ is less than or equal to $2$. At $\infty$,
by using $q=e^{2\pi iz}$, \eqref{(1.1)} can be written as
\begin{equation}
\left(q\frac{d}{dq}\right)^2y=-\frac{1}{4\pi^2}y''=-\frac{Q(z)}{4\pi^2}y.
\end{equation}
So \eqref{(1.1)} is Fuchsian at $\infty$ if and only if $Q$ is holomorphic at $\infty$.

Suppose that $z_0$ is a pole of $Q$.
The local exponents of \eqref{(1.1)} are $1/2\pm\kappa$, $\kappa\ge0$.
If the difference $2\kappa$ of the two local exponents is an integer,
then the ODE \eqref{(1.1)} might have a solution with a logarithmic
singularity at $z_0$. A singular point $z_0$ of \eqref{(1.1)} is
called \emph{apparent} if the local exponents are $1/2\pm\kappa$ with
$\kappa\in\frac{1}{2}\Z_{\ge0}$ and any solution of \eqref{(1.1)} has no
logarithmic singularity near $z_0$. In such a case, it is necessary
that $\kappa>0$. The ODE \eqref{(1.1)} or $Q$ is called
\emph{apparent} if \eqref{(1.1)} is apparent at any pole of $Q$ on
$\H$. Clearly, if \eqref{(1.1)} is apparent then the local monodromy
matrix at any pole is $\pm I$, where $I$ is the $2\x 2$ identity
matrix.

A solution $y(z)$ of \eqref{(1.1)} might be multi-valued. For
$\gamma\in\SL(2,\Z)$, $y(\gamma\cdot z)$ is understood as the analytic
continuation of $y$ along a path connecting $z$ and $\gamma\cdot
z$. Even though $y(\gamma\cdot z)$ is not well-defined, the slash
operator of weight $k$ ($k\in\Z$) can be defined in the usual way
by
\begin{equation}
  \left(y\big|_k\gamma\right)(z):=(cz+d)^{-k}y(\gamma\cdot z),\quad
  \gamma=\M abcd\in\SL(2,\Z),
\end{equation}
where $\gamma\cdot z=(az+b)/(cz+d)$. We have the well-known
Bol's identity \cite{Bol}
  $$
  \left(y\big|_{-1}\gamma\right)^{(2)}(z)=\left(y^{(2)}\big|_{3}
    \gamma\right)(z).
    $$
Hence, if $y(z)$ is a solution of \eqref{(1.1)}, then
$\left(y\big|_{-1}\gamma\right)(z)$ is also a solution of
\eqref{(1.1)}. Here $f^{(k)}(z)$ is the $k$-th derivative of $f(z)$.

%\noindent{\bf Bol's Theorem.} Suppose $y(z)$ is a solution of
%\eqref{(1.1)}. Then $\left(y\big|_{-1}\gamma\right)(z)$ is also a
%solution of \eqref{(1.1)}.\\

%We note that the classical Bol's theorem \cite{Bol} is stated only for
%holomorphic modular forms $Q$. Since the proof is a local argument and
%can be done by an elementary computation, it is easy to see the
%theorem holds also for meromorphic cases.

Suppose that \eqref{(1.1)} is \emph{apparent} and $y_i$, $i=1,2$, are two
independent solutions. Since the local monodromy matrix at any pole of
$Q$ is $\pm I$, the ratio $h(z)=y_2(z)/y_1(z)$ is well-defined and
meromorphic on $\H$. By Bol's identity, both
$\left(y_i\big|_{-1}\gamma\right)(z)$ are 
solutions of \eqref{(1.1)}, where $y_1(\gamma\cdot z)$ and
$y_2(\gamma\cdot z)$ are understood as the analytic continuation
of $y_1(z)$ and $y_2(z)$ along the same path connecting $z$ and
$\gamma\cdot z$. Note that since \eqref{(1.1)} is assumed to be
apparent, difference choices of paths from $z$ to $\gamma\cdot z$ only
result in sign changes in $y_1(\gamma\cdot z)$ and $y_2(\gamma\cdot
z)$. Therefore, there is a matrix $\rho(\gamma)$ in $\GL(2,\C)$ such that
\begin{equation}\label{equation: representation of weight -1}
\begin{pmatrix}
\left(y_1\big|_{-1}\gamma\right)(z)\\
\left(y_2\big|_{-1}\gamma\right)(z)
\end{pmatrix}=\pm\rho(\gamma)\begin{pmatrix}
y_1(z)\\
y_2(z)
\end{pmatrix}.
\end{equation}
Note that $\det \rho(\gamma)=1$ because the two Wronskians of fundamental solutions $\left( y_1\big|_{-1}\gamma,y_2\big|_{-1}\gamma\right)$ and $(y_1,y_2)$ are equal.
Hence $\rho$ is a homomorphism from $\SL(2,\Z)$ to $\rm{PSL}(2,\C)$.
In this
  paper, we call the homomorphism $\gamma\mapsto\pm\rho(\gamma)$
\emph{the Bol representation} associated to \eqref{(1.1)}.

There is an old problem in conformal geometry related to
\eqref{(1.1)}. The problem is to find a metric $ds^2$ with curvature
$1/2$ on $\H$ that is locally conformal to the flat metric and
invariant under the change $z\mapsto\gamma\cdot z$,
$\gamma\in\SL(2,\Z)$. Write $ds^2=e^u\abs{dz}^2$.
Below, we collect some basic results concerning the metric which will
be proved in Section $2$.
\begin{enumerate}
\item[(1)] The curvature condition is equivalent to saying that $u$
  satisfies the curvature equation \eqref{2.4}. Then
\begin{equation}
Q(z)=-\frac{1}{2}\left( u_{zz}-\frac{1}{2}u^2_z\right)
\end{equation}
is a meromorphic function.
\item[(2)]
 The invariant condition ensures that $Q$ is a meromorphic
modular form of weight $4$ with respect to $\SL(2,\Z)$ and holomorphic at $\infty$. Moreover, $Q(\infty)\le0$.

\item[(3)]
The metric might have  conic singularity at some $p\in\H$ with a conic
angle $\theta_p$, and the metric is smooth at $p$ if and only
$\theta_p=1$. Thus $Q$ has a pole at $p$ if and only $ds^2$ has a
conic singularity at $p$ (i.e., $\theta_p\neq 1$), provided that
$p\not\in\braces{\gamma\cdot i,\gamma\cdot\rho:\gamma\in\SL(2,\Z)}$,
where $i=\sqrt{-1}$ and $\rho=(1+\sqrt{-3})/2$.

\item[(4)]
Let $1/2\pm\kappa_p$, $\kappa_p>0$
be the local exponents at $p$ of \eqref{(1.1)} with this $Q$. 
Then $\theta_p=2\kappa_p/e_p$, where $e_p$ is the elliptic order of
$p$. Moreover, if $\kappa_p\in\frac{1}{2}\Z$ for any $p$, then \eqref{(1.1)} is
automatically apparent.
\end{enumerate}

We say the solution $u$ or the metric $e^u\abs{dz}^2$ \emph{realizes}
$Q$ or the associated ODE \eqref{(1.1)} is realized by $u$. We note
that for a given $Q$, finding a metric $e^u\abs{dz}^2$ realizing $Q$
is equivalent to solving the curvature equation \eqref{2.4} in Section
2 with the RHS being $4\pi\sum n_p\delta_p$, where $n_p=2\kappa_p-1$,
$\delta_p$ is the Dirac measure at $p\in\H$ and the summation runs
over all poles of $Q$ on $\H$. In particular, $\kappa_p\in\frac{1}{2}\N$,
if and only if the coefficient $n_p\in\N$, the set of positive
integers.

In view of this connection, throughout the paper, we assume that the ODE
\eqref{(1.1)} satisfy the following conditions ($\mb H_1$) or
($\mb H_2$).\\

\noindent{($\mb{H_1)}$} The ODE \eqref{(1.1)} is apparent with the local
exponents $1/2\pm\kappa_p$ at any pole $p$ of $Q$,
$\kappa_p\in\frac{1}{2}\N$, and $Q(\infty)\le0$. Denote the local
exponents at $\infty$ by $\pm\kappa_\infty$. Moreover, if
$p\not\in\braces{i,\rho}$, then $\kappa_p>1/2$.\\

Note that $Q(z)$ is smooth at $p$ if and only if $\kappa_p=1/2$, so
the requirement $\kappa_p>1/2$ means that that $Q(z)$ has a pole
at $p$. 
Note that by (4), the angle $\theta_\rho$ at $\rho$ is
$2\kappa_\rho/3$ and $\theta_i$ at $i$ is $\kappa_i$.\\

\noindent{($\mb{H_2}$)} The angles $\theta_\rho$ and $\theta_i$ are not integers.\\

Suppose $\kappa_\infty\not\in\frac{1}{2}\N$.  Then there is $r_\infty\in (0,1/2)$ such that
\begin{equation} \label{equation: r}
\text{either}\ \kappa_\infty\equiv r_\infty\mod 1\quad\text{or}\quad\kappa_\infty\equiv-r_\infty\mod 1.
\end{equation}

\begin{Theorem}\label{theorem: 1.1}
Suppose that \eqref{(1.1)} satisfies ($\mb H_1$), ($\mb
  H_2$), and $\kappa_\infty\not\in\frac{1}{2}\N$.
If $1/12<r_\infty<5/12$, then there is
an invariant metric of curvature $1/2$ realizing $Q$. Moreover, the metric is
unique. Conversely, if $Q$ is realized then $1/12\leq r_\infty\leq
5/12$.

Furthermore, assume that $r_\infty=1/12$ or $r_\infty=5/12$. Let
$\chi$ be the character of $\SL(2,\Z)$ determined by
$$
\chi(T)=e^{2\pi i/6}, \qquad \chi(S)=-1,
$$
where $T=\SM1101$ and $S=\SM0{-1}10$. Then there is an invariant metric
of curvature $1/2$ realizing $Q$ if and only there are two solutions
$y_1(z)$ and $y_2(z)$ of \eqref{(1.1)} such that $y_1(z)^2$ and
$y_2(z)^2$ are meromorphic modular forms of weight $-2$ with character
$\chi$ and $\overline\chi$, respectively, on $\SL(2,\Z)$.
%then the kernel of the Bol representation associated to \eqref{(1.1)} is the
%unique normal subgroup $\Gamma$ of $\SL(2,\Z)$ such that
%$\SL(2,\Z)/\Gamma$ is cyclic of order $6$.
%the commutator of $\SL(2,\Z)$ is a subgroup of $\ker\rho$ of index $2$. 
\end{Theorem}

\begin{Remark}
Let $\H^*=\H\cup\Q\cup\braces\infty$. Since $\SL(2,\Z)\backslash\H^\ast$ is
conformally diffeomorphic to the standard sphere $S^2$, Theorem
\ref{theorem: 1.1} can be formulated in terms of the existence of
metrics on $S^2$ with prescribed singularities at poles of $Q$ and
prescribed angle $\theta_p$ at each singular point $p$. In this sense,
Theorem \ref{theorem: 1.1} is a special case of a result of Eremenko
and Tarasov \cite{Eremenko-Tarasov}\footnote{We
  thank the referee for pointing out this and providing the
  reference.},
quoted as Theorem
\ref{theorem: ET} in the appendix. In the appendix, we give an
alternative and self-contained proof of their result in the form of
Theorem \ref{thm1} as it is
elementary and involves only straightforward matrix computation.
(In the notation of Theorem \ref{thm1}, Theorem
\ref{theorem: 1.1} corresponds to the case $\theta_1=1/2$,
$\theta_2=1/3$, and $\theta_3=2r_\infty$ or
$\theta_3=1-2r_\infty$, depending on whether $2r_\infty\le1/2$ or
$2r_\infty>1/2$.)
%Note that Theorem \ref{thm1} was first proved by
%Eremenko and Tarasov \cite[Theorem 2.5]{Eremenko-Tarasov}.
%We include our proof in the appendix as it is 

The threshold case $r_\infty\in\{1/12,5/12\}$ is more delicate. In
Section \ref{section: proof of theorem 1.1}, we provides examples
of existence and nonexistence of
an invariant metric with $r_\infty=1/12$. Our examples suggest that to
each $Q(z)$ with $r_\infty\in\{1/12,5/12\}$, one may associate a meromorphic
differential $1$-form $\omega$ of the second kind on a certain
elliptic curve $E$, and whether there exists an invariant metric
realizing $Q$ hinges on whether $\omega$ is exact, i.e., whether
$\omega$ is the identity element in the first de Rham cohomology group
of $E$. Also, in the nonexistence example, we find that the entries in
the monodromy matrices can be expressed in terms of periods or the
central value of the $L$-function of the elliptic curve
$y^2=x^3-1728$. We plan to study the threshold case in more details in
the future.
\end{Remark}
%The hypothesis says that
%there are three prescribed non-integer angles. As far as the
%authors know, there seems no existence result under such prescribed
%conditions.  See \cite{Eremenko, Eremenko2} for the results with less than or
%equal to three non-integer angles.
%Theorem \ref{theorem: 1.1} is new, provided that $m\geq 1$.

Motivated by Theorem \ref{theorem: 1.1}, we consider the datas given below.
\begin{equation}\label{equation: datas}
\begin{split}
&\text{A set of positive half-integers
}\ \kappa_\rho,\kappa_i, \kappa_j\in\N/2, j=1,2,\ldots,m,\\
&\text{such\ that}\ 2\kappa_\rho/3\not\in\N,\ \kappa_i\not\in\N;\ \text{a set of inequivalent points}\ p_j\in\H,\\
&j=1,2,\ldots, m;\text{ and  a positve number}\ \kappa_\infty.
\end{split}
\end{equation}

\begin{Definition}\label{definition: Q is equipped}
We say $Q$ is \emph{equipped with \eqref{equation: datas}} if
\begin{enumerate}
\item[(i)] $\braces{\rho,i,z_j:1\leq j\leq m}$ are the set of poles of $Q$;
\item[(ii)]The local exponents of $Q$ at $\rho, i, z_j$ are $1/2\pm\kappa_\rho$, $1/2\pm\kappa_i$ and $1/2\pm\kappa_j$, respectively;
\item[(iii)] $Q$ is apparent on $\H$; and
\item[(iv)] The local exponents at $\infty$ are $\pm\kappa_\infty$.
\end{enumerate}
\end{Definition}

\begin{Theorem}\label{theorem: existence of modular form Q with described local exponents and apparentness}
Given \eqref{equation: datas},
there are modular forms $Q$ of weight $4$ equipped with
\eqref{equation: datas}. Moreover, the number of such $Q$
is at most $\prod^m_{j=1}(2\kappa_j)$.
\end{Theorem}

To prove the theorem, we first show that there is a finite set of
polynomials such that the set of $Q(z)$ equipped with \eqref{equation:
  datas} is in a one-to-one correspondence with the set of common
zeros of the polynomial. Then the theorem follows immediately from the
clasical B\'ezout theorem. Note that Eremenko and Tarasov
\cite[Theorem 2.4]{Eremenko-Tarasov} has proved a stronger result,
which in our setting states that
for generic singular points $z_1,\ldots,z_m$, the number of $Q(z)$ is
precisely $\prod_{j=1}^m(2\kappa_j)$.
%\begin{Remark}
%  Again, we can formulate Theorem \ref{theorem: existence
%    of modular form Q with described local exponents and apparentness}
%  as a result on the sphere $S^2$. Our result is a little weaker
%  than \cite[Theorem 2.4]{Eremenko-Tarasov}, as Eremenko and Tarasov's
%  result implies that for generic singular points $z_1,\ldots,z_m$,
%  the number of modular forms $Q(z)$ is precisely
%  $\prod_{j=1}^m(2\kappa_j)$.\footnote{Again, we thank the referee for
%    providing the reference.} (It should be possible to
%  strengthen our argument to obtain the same result. We do not plan to
%  pursue in this direction. The interested reader should consult
%  \cite{Eremenko-Tarasov} and see how one might be able to get this
%  stronger result.)
%\end{Remark}

If the local exponents at $\infty$ are $\pm n/4$, $n$ is odd, then our second result asserts that there is a modular form of weight $-4$ coming from the equation. In the following, we use $T=\SM 1101$ and $S=\SM 0{-1}10$.

\begin{Theorem}\label{theorem: 1.2}
Suppose that ($\mb H_1$)  and ($\mb H_2$) hold and $\kappa_\infty=n/4$, $n$ a positive odd
integer. Then there is a constant $c\in\C$ such that
$F(z):=y_-(z)^2+cy_+(z)^2$ satisfies
$$
\left(F\big|_{-2}T\right)(z)=\left(F\big|_{-2}S\right)(z)=-F(z),
$$
where
$$
y_\pm(z)=q^{\pm n/4}\left(1+\sum_{j\geq 1}c_j^\pm q^j\right)
$$
are solutions of \eqref{(1.1)}.
\end{Theorem}

The constant $c$ is rational if all coefficients of $Q(z)/\pi^2$ in
the $q$-expansion are rational.
We conjecture $c$ is positive, but it is not proved yet.
Obviously, $F(z)^2$ is a modular form of weight $-4$ with respect to
$\SL(2,\Z)$. Let $\Gamma_2$ be the group generated by $T^2=\SM 1201$
and $ST=\SM0{-1}11$, which is an index $2$ subgroup of
$\SL(2,\Z)$. Then $F$ is a modular form of weight $-2$ on
$\Gamma_2$. This fact can help us to compute $c$ and $F(z)$
explicitly. For example, if $Q(z)=-\pi^2n^2E_4(z)/4$, then $F(z)$ is
holomorphic on $\H$, but with a pole of order $n$ at $\infty$
($\Gamma_2$ has only one cusp $\infty$ and two elliptic points of
order $3$). Thus it is not difficult to prove

\begin{Corollary}\label{corollary: 1.3}
Let $Q(z)=-\pi^2(n/2)^2E_4(z)$, where $n$ is a positive odd
  integer. Then there is a polynomial
$P_{n-1}(x)\in\Q[x]$ of degree $(n-1)/2$ such that
$$
F(z)=\frac{E_4(z)}{\Delta(z)^{1/2}}P_{n-1}(j(z)).
$$
Here $E_4$ and $E_6$ are the Eisenstein series of weight $4$ and $6$ on $\SL(2,\Z)$ respectively:
\begin{align*}
&E_4(z)=1+240\sum_{m=1}^\infty\frac{m^3q^m}{1-q^m}
=1+240\sum^\infty_{m=1}\left(\sum_{d|n}d^3\right)q^n, \quad q=e^{2\pi iz},\\
&E_6(z)=1-504\sum_{m=1}^\infty\frac{m^5q^m}{1-q^m}
=1-504\sum^\infty_{m=1}\left(\sum_{d|n}d^5\right)q^n,
\end{align*}
$\Delta(z)=(E_4(z)^3-E_6(z)^2)/1728=q-24q^2+\cdots$, and
$j(z)=E_4(z)^3/\Delta(z)$. 
\end{Corollary}

For small $n$, $P_{n-1}$ are shown in the
following list.
$$ \extrarowheight3pt
\begin{array}{c|ll} \hline\hline
  n & F & P_{n-1} \\ \hline
  1 & y_-^2+3(2^3y_+)^2 & 1 \\
  3 & y_-^2+3(2^{12}y_+)^2 & j-1536 \\
  5 & y_-^2+3(2^{18}7^1y_+)^2 & j^2-2240j+1146880 \\
  7 & y_-^2+3(2^{28}3^1y_+)^2 &
   j^3-3072j^2+2752512j-704643072\\
  9 & (7y_-)^2+3(2^{34}11^113^1y_+)^2 &
   \begin{split}
   49j^4&-192192j^3+253034496j^2-\\
   &-125954949120j+19346680184832
\end{split}\\
  \hline\hline
\end{array}
$$

In practice, it seems not easy to verify the apparentness at a singular point with local exponents $1/2\pm\kappa$, $\kappa\in\frac{1}{2}\N$.
Take a simple example
\begin{equation*}
\left(q\frac{d}{dq}\right)^2y=-\frac{1}{4\pi^2}y''=\left(\frac{n}{2}\right)^2E_4(z)y\quad \text{on}\ \H.
\end{equation*}
The local exponents at $\infty$ are $\pm n/2$. The standard
method to verify the apparentness at $\infty$ is to show that there
is a solution $y_-(z)$ having a $q$-expansion
$$
y_-(z)=q^{-n/2}\left(1+\sum_{j\geq 1}c_jq^j\right).
$$

Suppose $E_4(z)=\sum^\infty_{j\geq 0}b_jq^j$. Substituting the
$q$-expansion of $y_-$ and $E_4$ into the equation,
then the coefficient $c_j$ satisfies
\begin{equation}\label{1.7}
\left(\left(j-\frac{n}{2}\right)^2-\left(\frac{n}{2}\right)^2\right)c_j=\left(\frac{n}{2}\right)^2\sum_{k+\ell=j,\ \ell<j}b_kc_\ell.
\end{equation}
For $j=1,2,\ldots,n-1$, $c_j$ can be determined from $c_0=1$. However at
$j=n$, the LHS of \eqref{1.7} vanishes. Therefore, $\infty$ is apparent if and only
the RHS of \eqref{1.7} is $0$ at $j=n$. If $n$ is small, then it is
easy to check that the RHS of \eqref{1.7} is not $0$ at $j=n$. For a
general $n$, nevertheless, it seems not easy to see why it does not
vanish from the recursive relation \eqref{1.7}. Thus for a modular
ODE, the standard method is not efficient for this purpose. We need other ideas.

We consider
\begin{equation}\label{equation: DE 3}
y''(z)=\pi^2\left(rE_4(z)+s\frac{E_6(z)^2}{E_4(z)^2}+t\frac{E_4(z)^4}{E_6(z)^2}\right)y(z),
\end{equation}
where $r,s$ and $t$ are constant parameters. For simplicity, we denote
the potential of \eqref{equation: DE 3} by
$Q_3(z;r,s,t)$ or $Q_3(z)$ for short.
Modulo $\SL(2,\Z)$,
\eqref{equation: DE 3} has singularities only at $\rho$ and $i$ (recall that $E_4(z_0)=0$ if and only if $z_0$ is equivalent
to $\rho$ under $\SL(2,\Z)$ and $E_6(z_0)=0$ if and only if $z_0$ is
equivalent to $i$).
Assume the local exponents of \eqref{equation: DE 3} are
$1/2\pm\kappa_i$ at $i=\sqrt{-1}$ and $1/2\pm\kappa_\rho$ at
$\rho=(1+\sqrt{-3})/2$. Then it is easy to prove that $s=s_{\kappa_\rho}$, $t=t_{\kappa_i}$, where
\begin{equation}\label{equation: parameter s,t}
s_{\kappa_\rho}=\frac{1-4\kappa^2_\rho}{9},\qquad \text{and}\qquad t_{\kappa_i}=\frac{1-4\kappa^2_i}{4}.
\end{equation}
See Section 3 for the computation.

At $\infty$, the local exponents are $\pm\kappa_\infty$ if and only if
\begin{equation*}\label{1.9}
r+s_{\kappa_\rho}+t_{\kappa_i}=-(2\kappa_\infty)^2.
\end{equation*} 
In the following, we set the triple $\left(n_i,\ n_\rho,\ n_\infty\right)$ by
$$
(n_i,n_\rho,n_\infty)=\left(\kappa_i,\frac{2\kappa_\rho}{3},2\kappa_\infty\right).
$$

\begin{Theorem}\label{theorem: necessary and sufficient conditions for appartness at infinity}
%Suppose \eqref{equation: DE 3} satisfies ($\mb H_1$) and
%$\kappa_\infty\in\frac{1}{2}\N$.
The modular differential equation \eqref{equation: DE
  3} is apparent throughout $\H\cup\{\text{cusps}\}$ if and only if
the triple $(n_i,n_\rho, n_\infty)$ are positive integers satisfying
(i) the sum of these three integers is odd, and (ii) the sum of any
two of these integers is greater than the third.
Moreover, In such a case, the ratio of any two solutions is a
modular function on $\SL(2,\Z)$.
\end{Theorem}

For example, if 
$$
Q(z)=\pi^2\left(\frac{23}{36}E_4(z)-\frac{9n^2-1}{9}\frac{E_6(z)^2}{E_4(z)^2}-\frac{3}{4}\frac{E_4(z)^4}{E_6(z)^2}\right),\quad n\in\N,
$$
then we have $(n_i,n_\rho,n_\infty)=(1,n,n)$. By Theorem
\ref{theorem: necessary and sufficient conditions for appartness at
  infinity},
\eqref{equation: DE 3} is apparent throughout
$\H\cup\{\text{cusps}\}$. On the other hand, $\infty$
is not apparent for the ODE
\begin{equation*}\label{equation: DE 4}
y''(z)=-\pi^2n^2E_4(z)y(z).
\end{equation*}

%\begin{Remark}
%Indeed, if the hypothesis of Theorem \ref{theorem: necessary and
%sufficient conditions for appartness at infinity} holds and $n_i$,
%$n_\rho$ are positive integers, then the conditions (i) and (ii) hold
%automatically. See Theorem \ref{theorem: condition for Q be realized}
%in Section 5. 
%\end{Remark}
As discussed in \eqref{1.7}, it seems very difficult to verify ($\mb H_1$). So
we would like to present some examples to show how to verify the
condition ($\mb H_1$). The first example is
\begin{equation}\label{equation: DE in introduction}
y''(z)=\pi^2\left(rE_4(z)+s\frac{E_6(z)^2}{E_4(z)^2}\right)y(z),
\end{equation}
where $r,s$ are constant parameters.
For simplicity, we denote the potential of \eqref{equation: DE in
  introduction} by $Q_1(z;r,s)$ or $Q_1(z)$ for short. The singular
points modulo $\SL(2,\Z)$ is $\rho$ only.
If the local exponents are $1/2\pm \kappa_\rho$, then a simple calculation in Section 3 shows $s=s_{\kappa_\rho}$, where $s_{\kappa_\rho}$ is given in \eqref{equation: parameter s,t}.

\begin{Theorem}\label{theorem: 1.5}
  Let $\kappa_\rho\in\frac12\mathbb N$.
  \begin{enumerate}
    \item[(a)] Assume $3\nmid 2\kappa_\rho$. Then $Q_1(z;r,s)$ is apparent if $s=s_{\kappa_\rho}$ and any $r\in\C$.

      \item[(b)] Assume $3|2\kappa_\rho$.
      Then there exists a polynomial $P(x)\in\Q[x]$ of degree
      $2\kappa_\rho/3$ such that $Q_1(z;r,s)$ with 
      $s=s_{\kappa_\rho}$ is apparent if and only if $r$ is a
      root of $P(x)$. Moreover, $r$ satisfies
      \begin{equation}\label{equation: r+s=-(l+1/2)^2}
        r+s_{\kappa_\rho}=-\left(\ell+\frac{1}{2}\right)^2,
      \quad \text{where}\ \ell=0,1,2,\ldots,\frac{2\kappa_\rho}{3}-1.
      \end{equation}
  \end{enumerate}
      % if and only if $s=s_{\kappa_\rho}$ and any $r$. 
\end{Theorem}

Next, we consider the ODE
\begin{equation}\label{equation: DE 2 in introduction}
y''(z)=\pi^2\left(rE_4(z)+
t\frac{E_4(z)^4}{E_6(z)^2}
\right) y(z)\quad \text{on}\ \H,
\end{equation}
where $r$ and $t$ are constant parameters. For simplicity, the
potential of \eqref{equation: DE 2 in introduction} is denoted by
$Q_2(z;r,t)$ or $Q_2(z)$ for short. Similar to \eqref{equation: DE in
  introduction},
\eqref{equation: DE 2 in introduction} has local
exponents $1/2\pm \kappa_i$ at $i$ if and
only if $
t=t_{\kappa_i},
$
where $t_{\kappa_i}$ is given in \eqref{equation: parameter s,t}.

\begin{Theorem}\label{theorem: 1.7}
  Let $\kappa_i\in\frac12\mathbb N$.
\begin{enumerate}
\item[(a)] Assume $\kappa_i\in\frac{1}{2}+\Z_{\geq 0}$. Then
  \eqref{equation: DE 2 in introduction} 
   is apparent if and only if $t=t_{\kappa_i}$ and any
  $r\in\C$.
\item[(b)]
  Assume $\kappa_i\in\N$. Then there exists a polynomial
  $P(x)\in\Q[x]$ of degree $\kappa_i$ such that \eqref{equation: DE 2 in
    introduction} with $t=t_{\kappa_i}$ is apparent if and only
  if $r$ is a root of $P(x)$. Moreover, $r$ satisfies
  \begin{equation} \label{equation: r+t=-(l+1/3)^2}
  r+t_{\kappa_i}=-\left(\ell\pm\frac{1}{3}\right)^2, \quad\begin{cases} \ell=0,2,4,\ldots,\kappa_i-1,\quad&\text{if}\ \kappa_i\ \text{is\ odd},\\
  \ell=1,3,5,\ldots,\kappa_i-1,\qquad&\text{if}\ \kappa_i\ \text{is\ even}.\
  \end{cases}
  \end{equation}
\end{enumerate}
\end{Theorem}
We use the Frobenius method to prove Part (a) of Theorem \ref{theorem: 1.5}
and Theorem \ref{theorem: 1.7}. However, due to the modularity, our expansion of
functions are expanded in terms of powers 
of $w_\rho:=(z-\rho)/(z-\bar{\rho})$ and $w_i:=(z-i)/(z+i)$, not
powers of $z-\rho$ and $z-i$ as the 
standard method does. This kind of expansion has been used in \cite{Shimura-Maass}  and \cite{Zagier123}.
We will see in Section 3 that this type of expansions not only
simplifies computations greatly, but also obtains the degree of $P(x)$
in Theorem \ref{theorem: 1.5}(b) and Theorem \ref{theorem:
    1.7}(b) precisely.

We will present two proofs of \eqref{equation: r+s=-(l+1/2)^2}
in Theorem \ref{theorem: 1.5}(b) and \eqref{equation: r+t=-(l+1/3)^2} in Theorem \ref{theorem: 1.7}(b) in Section 4 and
Section 5. One is to apply Riemann's existence
theorem on compact Riemann surfaces. The other is to apply the
existence theorems of the invariant metrics with curvature $1/2$. This
geometric theorems are obtained by Eremenko
\cite{Eremenko,Eremenko2}. Hopefully, these methods are useful for
treating this kind of problems in modular differential equations.

The paper is organized as follows. In Section 2, we will discuss the
connection between the invariant metric $ds^2=e^u\abs{dz}^2$ of
curvature $1/2$ and modular ODEs, in particular, the relation among
the behavior of $u$ near cusps, angles and the local exponents of the
realized modular ODE by $u$. In Section 3, we will explain the
expansion of modular forms in terms of the natural coordinate
$w=(z-z_0)/(z-\bar{z}_0)$, and prove Theorem \ref{theorem: 1.5}(a) and
Theorem \ref{theorem: 1.7}(a). Both Theorem \ref{theorem: 1.5}(b) and Theorem
\ref{theorem: 1.7}(b) are proved in Section 4, and Theorem \ref{theorem: necessary and sufficient conditions for appartness at infinity}
is proved in Section 5. Finally, we will prove Theorem \ref{theorem: 1.1}
and Theorem \ref{theorem: 1.2} to complete the paper in Section 6 and Section 7 respectively.

\section{Curvature equations and the modular ODEs}
\subsection*{2.1.}
Let $M$ be a compact Riemann surface, $p\in M$, and $z$ be a complex
coordinate in an open neighborhood $U$ of $p$ with $z(p)=0$. We
consider the following curvature equation:
\begin{equation}\label{2.1}
4u_{z\bar{z}}+e^u=f\quad \text{on}\ U,
\end{equation}
where $f=4\pi\sum\alpha_i\delta_{p_i}$ is a sum of Dirac measures and
$0\neq\alpha_i>-1$. The assumption $\alpha_i>-1$ ensures that $e^u$ is
locally integrable in a neighborhood of $p_i$. The $L^1$-integrability
implies
\begin{equation}\label{equation: behavior of u near p}
u(z)=2\alpha_i\log\abs{z-p_i}+O(1)\quad \text{near}\ p_i.
\end{equation}
This is a general result from the elliptic PDE theory, see
\cite{Chen-Lin-dc, Chen-Lin-cpa}.

Let $w=w(z)$ be a coordinate change and set
\begin{equation}\label{equation: transformation of u}
\hat{u}(w)=u(z)-2\log\left|{\frac{dw}{dz}}\right|.
\end{equation}
Then $\hat{u}(w)$ also satisfies
\begin{equation*}
  4 \hat{u}_{w\bar{w}}+e^{\hat{u}}=\hat f, \qquad
  f=4\pi\sum\alpha_i\delta_{\hat p_i},
\end{equation*}
where $\hat p_i=w(p_i)$. In other words, $e^u\abs{dz}^2$ is invariant
under a coordinate change. Since $u$ has singularities at $p_i$, the
metric $ds^2=e^u\abs{dz}^2$ has a conic singularity at $p_i$. If $u$ is a
solution of \eqref{2.1}, then the metric $ds^2=e^u\abs{dz}^2$ has
curvature $1/2$ at any point $p\not\in\braces{p_i}$. Suppose that $M$ is
covered by $\braces{U_i}$ and $z_i$ is a coordinate in $U_i$. We call
the collection $\braces{u_i}$ to be a solution of \eqref{2.1} on $M$
if $u_i$ is a solution of \eqref{2.1} on $U_i$ for each $i$ and
satisfy the transformation law $u_j=u_i-2\log\left|\frac{dz_j}{dz_i}\right|$
on $U_i\cap U_j$.
 
Let $g$ be a metric of $M$ with the curvature $K$, and
the equation \eqref{2.1} on $M$ is equivalent to the curvature equation:
\begin{equation}\label{2.3}
\Delta_g\hat{u}+e^{\hat{u}}-K=4\pi\sum{\alpha}_i\delta_{p_i}\quad \text{on}\
 M,
\end{equation}where $\Delta_g$ is the Beltrami-Laplace operator of $(M,g)$.
We could normalize the metric $g$ such that the area of $M$ is equal
to $1$. In the case when $g$ has a constant curvature, \eqref{2.3} can
be written as
\begin{equation*}
 \Delta_g\hat{u}+\rho\left(\frac{e^{\hat{u}}}{\int e^{\hat{u}}}-1\right)=4\pi\sum{\alpha}_i(\delta_{p_i}-1)\quad \text{on}\ M.
\end{equation*}
This nonlinear PDE is often call \emph{a mean field equation} in
analysis. See \cite{Chai-Lin-Wang, Chen-Lin-cpa, Chen-Lin-dc, Chen-Lin-spectrum, Chen-Lin-weight2, Chen-Kuo-Lin-simplezero}
and \cite{Lin-green-function, Lin-Wang-elliptic, Lin-Wang-mean-field} for the recent development of mean field equations.

In this paper, we consider the compact Riemann surface that is the
quotient of $\H^*:=\H\cup\Q\cup\braces{\infty}$ by a finite index
subgroup $\Gamma$ of $\SL(2,\Z)$, and the equation \eqref{2.1} is
defined on the upper half space $\H$:
\begin{equation}\label{2.4}
4u_{z\bar{z}}+e^u=4\pi\sum\alpha_i\delta_{p_i}\quad \text{on}\ \H,
\end{equation}
where the RHS is invariant under the action of $\Gamma$, i.e., the set
$\braces{p_i}$ is invariant under the action of $\Gamma$ and
$\alpha_i=\alpha_j$ if $p_i=\gamma\cdot p_j$ for some
$\gamma\in\Gamma$. The transformation law \eqref{equation:
  transformation of u} for coordinate change is equivalent to asking
$u$ to satisfy
\begin{equation}\label{2.5}
u(\gamma z)=u(z)+4\log\abs{cz+d},\quad \forall\gamma=\M abcd\in\Gamma.
\end{equation}
Let $s$ be a cusp of $\Gamma$ and $\gamma\in\SL(2,\Z)$ be a matrix
such that $\gamma\cdot\infty=s$. Then we define $u_\gamma$ by
$$
u_\gamma(z):=u(\gamma\cdot z)-4\log\abs{cz+d}.
$$
Thus, $u$ is required to satisfy the following behavior near $s$: there
is $\alpha_s>0$ such that
\begin{equation}\label{2.6}
 e^{u_\gamma(z)}=\abs{q_N}^{4\alpha_s}(c+o(1)),\quad q_N=e^{2\pi i z/N},\ c>0,
 \end{equation}
where $N$ is the width of the cusp $s$ and
$o(1)\rightarrow 0$ as $q_N\rightarrow 0$. Given the RHS of
\eqref{2.4} and a positive $\alpha_s$ at the cusp $s$, we ask for a
solution $u$ of \eqref{2.4} satisfying \eqref{2.5} and \eqref{2.6} at
any cusp.

The conic angle $\theta$, defined at a singularity $p_i$ or a cusp
$s$, is an important geometric quantity. Suppose that a metric $ds^2$,
conformal to the flat metric $\abs{dz}^2$, has a conic singularity at
$p$, and $w$ is a coordinate near $p$ with $w(p)=0$. If
\begin{equation}\label{equation: ds^2=}
ds^2=\abs w^{2(\theta-1)}(c+o(1))\abs{dw}^2,\quad c>0,
\end{equation}
then we call $\theta$ the \emph{angle} at $p$, and $2\pi\theta$ the \emph{total angle} at $p$.
Since $ds^2$ is
required to have a finite area, the angle $\theta$ is always
\emph{positive}. Note that $ds^2$ is smooth (as a metric) at $p$ if
and only if $\theta=1$.

Next, we want to calculate the angles of $ds^2=e^u\abs {dz}^2$, if $u$
is a solution of \eqref{2.4}. Note that $z$ is not a coordinate of $M$
if $p_i$ is an elliptic point of order $e_i>1$. Indeed,
$w=(z-p_i)^{e_i}$ is the local coordinate near $p_i$. For simplicity,
we denote $z-p_i$ by $z$ ($z(p_i)=0$). By \eqref{equation: behavior of
  u near p}, we have $u(z)=2\alpha_i\log\abs z+O(1)$, i.e.,
$e^{u(z)}=\abs{z}^{2\alpha_i}(c_0+o(1))$, $c_0>0$.
Then
$$
e^{u(z)}\abs{dz}^2=\abs{w}^{(2\alpha_i+2)/e_i-2}(d+o(1))\abs{dw}^2,\quad d>0.
$$
By \eqref{equation: ds^2=}, we have
\begin{equation}\label{2.7}
\theta_i=\frac{\alpha_i+1}{e_i}.
\end{equation}
At a cusp $s$, the coordinate is $q_N=e^{2\pi i z/N}$, where $N$ is
the width of the cusp $s$. By \eqref{2.6},
$$
e^{u_{\gamma}(z)}\abs{dz}^2=\abs{q_N}^{4\alpha_s-2}(c+o(1))\abs{dq_N}^2,\quad c>0.
$$
So the angle $\theta_s$ at $s$ is
\begin{equation}\label{2.8}
 \theta_s=2\alpha_s.
\end{equation}
 
\subsection*{2.2. Integrability and modular differential equations}
Equation \eqref{2.4} is also known as an
integrable system. There are two important features related to the
integrability. One is that
\begin{equation}\label{equation: uzz-uz/2 mero}
Q(z):=-\frac{1}{2}\left( u_{zz}-\frac{1}{2}u^2_z\right)\quad\text{is\ a\ meromorphic\ function},
\end{equation}
because $Q(z)_{\bar{z}}=-\frac{1}{2}(u_{z\bar{z}z}-u_{z\bar{z}}u_z)=0$ by \eqref{2.4}.
\begin{Lemma}\label{lemma: 2.1}
Each $p_i$ is a double pole of $Q(z)$ with the expansion $\frac{\alpha_i}{2}\left(\frac{\alpha_i}{2}+1\right)(z-p_i)^{-2}+O\left((z-p_i)^{-1}\right)$.
\end{Lemma}

\begin{proof}
Since $u(z)=2\alpha_i\log\abs{z-p_i}+O(1)$ near $p_i$, we have
$u_z(z)=\alpha_i(z-p_i)^{-1}+O(1)$ and
$u_{zz}(z)=-\alpha_i(z-p_i)^{-2}+O\left((z-p_i)^{-1}\right)$. Then the lemma
follows immediately.
\end{proof}

On the other hand, the Liouville theorem asserts that locally any
solution $u$ can be expressed as
\begin{equation}\label{2.10}
u(z)=\log\frac{8\abs{h'(z)}^2}{\left(1+\abs{h(z)}^2\right)^2},
\end{equation}
where $h(z)$ is a meromorphic function. Recall the Schwarz derivative
\begin{equation}\label{3.3}
  \braces{h,z}=\left(\frac{h''}{h'}\right)'
  -\frac{1}{2}\left(\frac{h''}{h'}\right)^2.
\end{equation}
Note that the Schwarz derivative can be used to recover $h$ from $u$. Indeed, a direct computation from \eqref{2.10} yields that
\begin{equation}\label{2.12}
\braces{h,z}=-2Q(z).
\end{equation}
See \cite{Chai-Lin-Wang, Lin-green-function, Lin-Wang-elliptic, Lin-Wang-mean-field} for the detail of the proofs \eqref{2.10}--\eqref{2.12}.
The meromorphic function $h$ is called a \emph{developing map} for the
solution $u$. Any two developing maps $h_i$, $i=1,2$, of $u$ have the
same Schwarz derivative by \eqref{2.12}, thus they
can be connected by a M\"obius transformation,
\begin{equation}\label{2.13}
h_2(z)=\frac{ah_1(z)+b}{ch_1(z)+d},\quad \M abcd\in\SL(2,\C).
\end{equation}
By \eqref{2.10}, we obtain
\begin{equation}\label{3.6}
  \frac{\abs{h'_1(z)}^2}{\left(1+\abs{h_1(z)}^2\right)^2}
  =\frac{\abs{h_2'(z)}^2}{\left(1+\abs{h_2(z)}^2\right)^2},
\end{equation}
which implies that the matrix $\SM abcd$ is an unitary matrix.

Next, we recall the classical Hermite theorem, see \cite{Whittaker-Watson}.

\begin{theorem}
Let $y_i$, $i=1,2$, be  two independent solutions of
$$
y''=Q(z)y.
$$
Then the ratio $h(z)=y_2(z)/y_1(z)$ satisfies $\braces{h,z}=-2Q(z)$.
\end{theorem} 

Let $Q(z)$ be the meromorphic function \eqref{equation: uzz-uz/2 mero}
obtained from the solution $u$. Consider the ODE
\begin{equation}\label{equation: y''=Qy in section 2}
y''=Q(z)y.
\end{equation}
Then \eqref{equation: uzz-uz/2 mero} and the Hermite theorem together imply that $h(z)$ is a ratio of two solutions of \eqref{3.6}.

\begin{Theorem}\label{theorem: 2.1}
Suppose $u$ is a solution of \eqref{2.4}. Then \eqref{equation: y''=Qy in section 2} satisfies ($\mb H_1$) and the followings hold.

\begin{enumerate}
\item[(a)] 
The function $Q(z)$ is a meromorphic modular form of weight $4$ with
respect to $\Gamma$ and holomorphic at any cusp. Moreover, at a cusp
$s$, $Q(s)<0$.
\item[(b)] \eqref{equation: y''=Qy in section 2} is Fuchsian and the local exponents of \eqref{equation: y''=Qy in section 2}
at $p_i$ are $-\alpha_i/2$, $\alpha_i/2+1
$, and $\pm\alpha_s$ at a cusp.
\item[(c)]
If $\alpha_i\in\N$ for all $i$, then \eqref{equation: y''=Qy in section 2} is apparent.

\end{enumerate}

\end{Theorem}

\begin{proof}
(a) By the chain rule, we have
\begin{align*}
(u\ccirc\gamma)_z(z)&=u_z(\gamma\cdot z)(cz+d)^{-2},\\
  (u\ccirc\gamma)_{zz}(z)&=u_{zz}(\gamma\cdot z)(cz+d)^{-4}
                           -u_z(\gamma\cdot z)\frac{2c}{(cz+d)^3}.
\end{align*}
Thus
\begin{align*}
  (u\ccirc \gamma)_{zz}-\frac{1}{2}(u\ccirc \gamma)^2_z
  &=\left(u_{zz}(\gamma\cdot z)-\frac{1}{2}u^2_z(\gamma\cdot z)\right)\\
  &\qquad\x(cz+d)^{-4}-u_z(\gamma\cdot z)\cdot\frac{2c}{(cz+d)^3}.
\end{align*}
On the other hand, the transformation law \eqref{2.5} yields
\begin{align*}
(u\ccirc \gamma)_z(z)&=u_z(z)+\frac{2c}{(cz+d)}, \\
(u\ccirc\gamma)_{zz}&=u_{zz}-\frac{2c^2}{(cz+d)^2}.
\end{align*}
Hence, we have
\begin{align*}
(u\ccirc\gamma)_{zz}-\frac{1}{2}(u\ccirc\gamma)^2_z&=\left(u_{zz}-\frac{1}{2}u^2_z\right)-u_z(z)\cdot\frac{2c}{(cz+d)}-\frac{4c^2}{(cz+d)^2}\\
&=\left(u_{zz}-\frac{1}{2}u^2_z\right)-\frac{2c}{(cz+d)}(u\ccirc\gamma)_z(z).
\end{align*}
Since 
$$
\frac{-2c}{(cz+d)}(u\ccirc\gamma)_z=\frac{-2c}{(cz+d)^3}u_z(\gamma\cdot z),
$$
we find that $Q:=-\frac{1}{2}\left(u_{zz}-\frac{1}{2}u^2_z\right)$ satisfies
$$
Q(\gamma\cdot z)=Q(z)\cdot(cz+d)^4.
$$
This proves the modularity of $Q$. 

To prove the holomorphy of $Q$ at cusps, without loss of generality,
we may assume that the cusp $s$ is $\infty$. Then $q_N=e^{2\pi iz/N}$ is
the local coordinate near $\infty$, where $N$ is the width of
the cusp $\infty$. By the transformation law of coordinate changes, the
solution $\hat{u}$ in terms of $q_N$ should be expressed by
$\hat{u}(q_N)=u(z)-2\log\left|\frac{dq_N}{dz}\right|$.
Thus,
$$
e^{\hat{u}(q_N)}=\frac{8\abs{h'(z)}^2}{\left(1+\abs{h(z)}^2\right)^2}
\left|\frac{dq_N}{dz}\right|^2
=8\left|\frac{d}{dq_N}h(z)\right|^2\left(1+\abs{h(z)}^2\right)^{-2}.
$$
Hence the developing map $h(z)=\hat{h}(e^{2\pi iz/N})=\hat{h}(q_N)$, where $q_N=e^{2\pi
  iz/N}$. 
 Note that
\begin{align*}
\braces{h,z}&=\braces{\hat{h},q_N}\left(\frac{dq_N}{dz}\right)^2+\braces{q_N,z}\\
&=\braces{\hat{h},q_N}q^2_N\left(\frac{-4\pi^2}{N^2}\right)+\frac{2\pi^2}{N^2}.
\end{align*}
Since
$$
-\frac{1}{2}\braces{\hat{h},q_N}=\hat{u}_{q_Nq_N}-\frac{1}{2}\hat{u}_{q_N}^2=\frac{\alpha}{2}\left(\frac{\alpha}{2}+1\right)q^{-2}_N+O\left(q^{-1}_N\right),
$$
where $\alpha=\theta-1$, $\theta$ is the angle at $\infty$, we have
\begin{equation*}
  \begin{split}
\lim_{\Im z\rightarrow\infty}Q(z)&=-\frac{\pi^2}{N^2}
\left(1+\frac{4\alpha}{2}\left(\frac{\alpha}{2}+1\right)\right)=-\frac{\pi^2}{N^2}(1+\alpha)^2<0,
\end{split}
\end{equation*}
because $\alpha>-1$. This proves Part (a).

Part (b) is a consequence of Lemma \ref{lemma: 2.1}.

For Part (c), if $\alpha_i\in\N$ then the local exponents $-\alpha_i/2$ and $\alpha_i/2+1$ can be written as
$1/2\pm \kappa_i$, $\kappa_i=(\alpha_i+1)/2\in\frac{1}{2}\N$ and by the Liouville theorem \eqref{2.10}, we see easily that $h(z)$ can not have a logarithmic singularity at $p_i$. The fact that $h(z)$ is a ratio of two solutions of \eqref{equation: y''=Qy in section 2} implies any solution of \eqref{equation: y''=Qy in section 2} has no logarithmic singularity. This proves Part (c).
\end{proof}

Together with the Liouville theorem, we have
\begin{Proposition}\label{proposition: Q can be realized}
Suppose $Q$ is a meromorphic modular form of weight $4$ on $\SL(2,\Z)$. If there are two independent solutions $y_1$ and $y_2$ of \eqref{equation: y''=Qy in section 2} such that $h(z)=y_2(z)/y_1(z)$ satisfies $h(\gamma z)=\frac{a h(z)+b}{ch(z)+d}$ for some unitary matrix $\SM abcd$ depending on $\gamma$, for any $\gamma\in\SL(2,\Z)$, then $Q$ can be realized. 
\end{Proposition}
\begin{proof}
Let
$u(z)=\log\frac{8\abs{h'(z)}^2}{\left(1+\abs{h(z)^2}\right)^2}$. Since
$h(z)$ is unitary, $u(z)$ is well-defined on $\H$ and satisfies
\eqref{2.5}. Further, the Liouville theorem says that $u(z)$ satisfies
\eqref{2.4}.
\end{proof}

\subsection*{2.3. Examples}
In this subsection, we will give some examples to indicate how to
determine $Q$ provided that the RHS of \eqref{2.4} is
known and $\alpha_\infty$ is given at $\infty$. Here,
$\Gamma=\SL(2,\Z)$.\\ 

\noindent{\bf Example 1.}
Assume that the RHS of \eqref{2.4} is equal to $0$. Then
$Q:=-\frac{1}{2}\left(u_{zz}-\frac{1}{2}u^2_z\right)$
is a holomorphic modular form of weight $4$. Thus,
\begin{equation}
Q(z)=\pi^2rE_4(z).
\end{equation}
Since $\pm\alpha_\infty$ are the local exponents of \eqref{(1.1)} at
$\infty$, we have
$r=-4\alpha^2_\infty$. Thus, $Q$ is uniquely determined. Note that at
$\infty$, the angle $\theta_\infty$ is equal to $2\alpha_\infty$.\\

\noindent{\bf Example 2.}
Assume that the RHS of \eqref{2.4} is $4\pi n\sum\delta_p$, where the
summation is over $\gamma\cdot\rho$ for every $\gamma\in\Gamma$. Then
$Q$ is a meromorphic modular form of weight $4$ whose poles occur at
$\gamma\cdot\rho$ and the order is $2$. Thus, $E_4(z)^2
Q(z)$ is holomorphic a modular form of wright $12$, and then
$$
Q(z)=\pi^2\left(rE_4(z)+s\frac{E_6(z)^2}{E_4(z)^2}\right),
$$ where we recall that the graded ring of modular forms on $\SL(2,\Z)$ is generated by $E_4(z)$ and $E_6(z)$.
By Theorem \ref{theorem: 2.1}, the local exponents at $\rho$ are $-n/2$ and $n/2+1$, which implies 
$\kappa
_\rho=(n+1)/2$, $s=(1-4\kappa^2_\rho)/9$,
and $-(r+s)/4=\alpha^2_\infty$. Thus $Q$ is uniquely determined. Moreover, the angles $\theta_j$ in this example are $\theta_i=1/2$, $\theta_\rho=(n+1)/3$ and $\theta_\infty=2\alpha_\infty$.\\

 \noindent{\bf Example 3.}
Assume that the RHS of \eqref{2.4} is equal to $4\pi n\sum\delta_p$,
where the summations is over $\gamma\cdot i$ for any
$\gamma\in\Gamma$. Reasoning as Example 2, we have
\begin{equation}
Q(z)=\pi^2\left(rE_4(z)+t\frac{E_4(z)^4}{E_6(z)^2}\right).
\end{equation}
By Theorem \ref{theorem: 2.1}, we have
$$
 \kappa_i=\frac{n+1}{2},\quad t=\frac{1-4\kappa^2_i}{4},
 \quad \text{and}\quad r+t=-4\alpha^2_\infty.
$$
Thus $Q$ is uniquely determined.
Moreover, we have $\theta_i=(n+1)/2$, $\theta_\rho=1/3$, and $\theta_\infty=2\alpha_\infty$.
 \\

 \noindent{\bf Example 4.}
Assume the RHS of \eqref{2.4} is
$4\pi\left(n\sum_{p_1}\delta_{p_1}+m\sum_{p_2}\delta_{p_2}\right)$,
where $p_1,p_2$ run over zeros of $E_4(z)$ and $E_6(z)$,
respectively. Then
 \begin{equation}
Q(z)=\pi^2\left(rE_4(z)+s\frac{E_6(z)^2}{E_4(z)^2}+t\frac{E_4(z)^4}{E_6(z)^2}\right).
\end{equation}
The conditions on the local exponents at $\rho$, $i$ and $\infty$ yield that
\begin{align*}
&s=\frac{1-4\kappa_\rho^2}{9},\quad \kappa_\rho=\frac{n+1}{2};\quad t=\frac{1-4\kappa^2_i}{4},\quad \kappa_i=\frac{m+1}{2};\\
&r+s+t=-4\alpha^2_\infty.
\end{align*}
Then $Q$ is uniquely determined. Moreover, $\theta_1=(m+1)/2$,
$\theta_2=(n+1)/3$ and $\theta_\infty=2\alpha_\infty$.

\subsection*{2.4 Eremenko's theorem}
A. Eremenko \cite{Eremenko, Eremenko2} gave a necessary and sufficient
conditions of the angles $\theta_i$, $1\leq i\leq 3$, at the three
singular points $i,\rho,\infty$ for the existence of $u$ of \eqref{2.4}-\eqref{2.6}.\\

When one of angles is an integer, the following conditions are required.\\

\noindent{\bf (A)} If only one (say $\theta_1$) of angles is an
integer, then either $\theta_2+\theta_3$ or $\abs{\theta_2-\theta_3}$
is an integer $m$ of opposite parity to $\theta_1$ with
$m\leq\theta_1-1$. If all the angles are integers, then (1)
$\theta_1+\theta_2+\theta_3$ is odd, and (2)
$\theta_i<\theta_j+\theta_k$ for $i\neq j\neq k$.\\

\noindent{\bf Eremenko's theorem.}
If one of $\theta_j$ is an integer, then a necessary and sufficient condition for the existence of a conformal metric of 
positive constant
curvature on the sphere with three conic singularities of angles
$\theta_1$, $\theta_2$, $\theta_3$ ($\theta_j\neq 1$, $1\leq j\leq
3$), is that $\braces{\theta_1,\theta_2,\theta_3}$ satisfies
(A). Moreover, if (A) holds and there is only one integral angle, then
the metric is unique.

\section{Expansions of Eisenstein series at $\rho$ and $i$}

The $q$-expansion of a modular form $f(z)$, i.e., the expansion of
$f(z)$ with respect to the local parameter $q$ at the cusp $\infty$,
is frequently studied and is of great significance in many problems in
number theory. Here we shall review properties of series expansions of
modular forms at a point $z_0\in\H$, other than the cusp $\infty$.

\begin{Definition} Let $\Gamma$ be a Fuchsian subgroup of the first
  kind of $\SL(2,\R)$. Let $f(z)$ be a meromorphic modular form of
  weight $k$ on $\Gamma$. Given $z_0\in\H$, let
  $$
  w=w_{z_0}(z)=\frac{z-z_0}{z-\overline z_0}.
  $$
  The expansion of the form
  \begin{equation} \label{equation: power series expansion}
  f(z)=(1-w)^k\sum_{n\ge n_0}\frac{b_n}{n!}w^n
  \end{equation}
  is called the \emph{power series expansion} of $f$ at $z_0$.
\end{Definition}

One advantage of this expansion is that its coefficients $b_n$ have a
simple expression in terms of the Shimura-Maass derivatives of $f$. To
state the result, we recall that if $f:\H\to\C$ is said to be \emph{nearly
  holomorphic} if it is of the form
$$
f(z)=\sum_{d=0}^n\frac{f_d(z)}{(z-\overline z)^d}
$$
for some holomorphic functions $f_d$. If $k$ is an integer and
$f:\H\to\C$ is a nearly holomorphic function such that
$$
f\left(\frac{az+b}{cz+d}\right)=(cz+d)^kf(z)
$$
for all $\SM abcd\in\Gamma$ and each $f_d$ is holomorphic at every
cusp, then we say $f$ is a \emph{nearly holomorphic modular form} of
weight $k$ on $\Gamma$.

For a nearly holomorphic function $f:\H\to\C$, we define its
\emph{Shimura-Maass} derivative of weight $k$ by
$$
(\partial_kf)(z):=\frac1{2\pi i}\left(f'(z)+\frac{kf(z)}{z-\overline
    z}\right).
$$
We have the following important properties of Shimura-Maass
derivatives.

\begin{Lemma}[{\cite[Equations (1.5) and (1.8)]{Shimura-Maass}}]
\label{lemma: Maass}
For any nearly holomorphic functions $f,g:\H\to\C$, any integers $k$
and $\ell$, and any $\gamma\in\GL^+(2,\R)$, we have
$$
  \partial_{k+\ell}(fg)=(\partial_k f)g+f(\partial_\ell g)
$$
and
$$
  \partial_k\left(f\big|_k\gamma\right)=(\partial_k f)\big|_{k+2}\gamma.
$$
\end{Lemma}

\begin{Remark} The second property in the lemma implies that if $f$ is
a nearly holomorphic modular form of weight $k$ on $\Gamma$, then
$\partial_kf$ is a nearly holomorphic form of weight 
$k+2$ on $\Gamma$.
\end{Remark}

Set also
$$
\partial_k^nf=\partial_{k+2n-2}\ldots\partial_kf.
$$
Then the coefficients $b_n$ in \eqref{equation: power series
  expansion} has the following expression.

\begin{Proposition}[{\cite[Proposition 17]{Zagier123}}]
  \label{proposition: power series coefficients}
  If $f(z)$ is a
  holomorphic modular form of weight $k$ on $\Gamma$, then the
  coefficients $b_n$ in \eqref{equation: power series expansion}
  are
  $$
  b_n=(\partial^n_kf)(z_0)(-4\pi\,\Im z_0)^n
  $$
  for $n\ge 0$. That is, we have
  $$
  f(z)=(1-w)^k\sum_{n=0}^\infty\frac{(\partial^n_kf)(z_0)
    (-4\pi\,\Im z_0)^n}{n!}w^n.
  $$
\end{Proposition}

Note that there is a misprint in Proposition 17 \cite{Zagier123}. The
proof of the proposition shows that $b_n=(\partial^nf)(z_0)(-4\pi
\,\Im z_0)^n$, but the statement misses the minus sign.

We will use these properties of power series expansions of modular
forms to show that the apparentness of \eqref{(1.1)} at a point $z_0$
will imply the apparentness at $\gamma z_0$ for all
$\gamma\in\SL(2,\Z)$. We first prove two lemmas. The first lemma
relates the power series expansion of a meromorphic modular form at
$z_0$ to that at $\gamma z_0$.

\begin{Lemma} \label{lemma: other points}
  Assume that $f$ is a meromorphic modular form of weight $k$ on
  $\SL(2,\Z)$. Assume that the power series expansion of $f$ at
  $z_0\in\H$ is
  $$
  f(z)=(1-w)^k\sum_{n=n_0}^\infty a_nw^n, \qquad
  w=w_{z_0}(z)=\frac{z-z_0}{z-\overline z_0}.
  $$
  For $\gamma=\SM abcd\in\SL(2,\Z)$, let
  $\wt w=w_{\gamma z_0}(z)=(z-\gamma z_0)/(z-\gamma\overline{z}_0)$.
  Then the power series expansion of $f$ at $\wt z_0$ is
  $$
    (cz_0+d)^k(1-\wt w)^k\sum_{n=n_0}^\infty
    a_n\left(\frac{cz_0+d}{c\overline z_0+d}
      \right)^n\wt w^n.
  $$
\end{Lemma}

\begin{proof} Since every meromorphic modular form on $\SL(2,\Z)$ can
  be written as the quotient of two holomorphic modular forms on
  $\SL(2,\Z)$, it suffices to prove the lemma under the assumption
  that $f$ is a holomorphic modular form.
  
  According to Proposition \ref{proposition: power series
    coefficients}, the power series expansions of $f$ at $z_0$ and at
  $\gamma z_0$ are
  $$
  (1-w)^k\sum_{n=0}^\infty\frac{(\partial^n_kf)(z_0)(-4\pi\,\Im
    z_0)^n}{n!}w^n
  $$
  are
  $$
  (1-\wt w)^k\sum_{n=0}^\infty\frac{(\partial^n_kf)(\gamma z_0)(-4\pi\,\Im
    \gamma z_0)^n}{n!}\wt w^n,
  $$
  respectively. Since $\partial^nf(z)$ is modular of weight $k+2n$
  (see the remark following Lemma \ref{lemma: Maass}), we have
  $$
  (\partial^nf)(\gamma z_0)=(cz_0+d)^{k+2n}(\partial^nf)(z_0).
  $$
  Also,
  \begin{equation} \label{equation: Im}
  \Im\gamma z_0=\frac{\Im z_0}{|cz_0+d|^2}.
  \end{equation}
  Thus, if the power series expansion of $f$ at $z_0$ is
  $$
    (1-w)^k\sum_{n=0}^\infty\frac{b_n}{n!}w^n,
  $$
  then that of $f$ at $\gamma z_0$ is
  \begin{equation*}
    \begin{split}
    &(1-\wt w)^k\sum_{n=0}^\infty\frac{b_n}{n!}
    \frac{(cz_0+d)^{k+2n}}{|cz_0+d|^{2n}}\wt w^n \\
    &\qquad=(cz_0+d)^k(1-\wt w)^k\sum_{n=0}^\infty
    \frac{b_n}{n!}\left(\frac{cz_0+d}{c\overline z_0+d}\right)^n
    \wt w^n.
    \end{split}
  \end{equation*}
  This proves the lemma.
\end{proof}

The next lemma expresses $y''(z)$ in terms of $w$.
%The following computation is useful for studying modular differential
%equations.

\begin{Lemma} \label{lemma: df/dw}
  Let $z_0\in\H$ and set $w=w_{z_0}(z)=(z-z_0)/(z-\overline z_0)$.
  If
  $$
  y(z)=\frac1{1-w}\sum_{n=0}^\infty a_nw^{\alpha+n}
  $$
  for some real number $\alpha$, then
  \begin{equation*}
  \frac{d^2}{dz^2}y(z)=\frac{(1-w)^3}{(z_0-\overline z_0)^2}
  \sum_{n=0}^\infty a_n(\alpha+n)(\alpha+n-1)w^{\alpha+n-2}.
  \end{equation*}
\end{Lemma}

\begin{proof} We first note that
  $$
  1-w=\frac{z_0-\overline z_0}{z-z_0}
  $$
  and hence
  \begin{equation} \label{equation: w'}
  \frac{dw}{dz}=\frac{z_0-\overline z_0}{(z-z_0)^2}
  =\frac{(1-w)^2}{z_0-\overline z_0}, \quad
  \frac{d^2w}{dz^2}=-2\frac{z_0-\overline z_0}{(z-z_0)^3}
  =-\frac{2(1-w)^3}{(z_0-\overline z_0)^2}.
  \end{equation}
  Let $g(w)=\sum a_nw^{\alpha+n}$. We compute that
  $$
  \frac{dy}{dz}=\left(\frac1{(1-w)^2}g(w)
    +\frac1{1-w}\frac{dg(w)}{dw}\right)\frac{dw}{dz}
  $$
  and
  \begin{equation*}
    \begin{split}
  \frac{d^2y}{dz^2}&=\left(
    \frac2{(1-w)^3}g(w)+\frac2{(1-w)^2}\frac{dg(w)}{dw}
    +\frac1{1-w}\frac{d^2g(w)}{dw^2}\right)
    \left(\frac{dw}{dz}\right)^2 \\
  &\qquad+\left(\frac1{(1-w)^2}g(w)+\frac1{1-w}\frac{dg(w)}{dw}\right)
  \frac{d^2w}{dz^2}.
    \end{split}
  \end{equation*}
  Using \eqref{equation: w'}, we reduce this to
  $$
  \frac{d^2y}{dz^2}=\frac{(1-w)^3}{(z_0-\overline z_0)^2}
  \frac{d^2g(w)}{dw^2}.
  $$
  This proves the lemma.
\end{proof}

\begin{Proposition} \label{proposition: same exponents}
  Suppose that $Q$ is a meromorphic modular form of weight $4$ with
  respect to $\SL(2,\Z)$ such that \eqref{(1.1)} is Fuchsian. Let $z_0$ be a
  pole of $Q$. Then the local exponents of \eqref{(1.1)} at $\gamma
  z_0$ are the same for all $\gamma\in\SL(2,\Z)$. Also, if
  \eqref{(1.1)} is apparent at $z_0$, 
  then it is apparent at $\gamma z_0$ for all $\gamma\in\SL(2,\Z)$.
\end{Proposition}

\begin{proof} Let $\gamma=\SM
  abcd\in\SL(2,\Z)$, $w=(z-z_0)/(z-\overline z_0)$, and $\wt
  w=(z-\gamma z_0)/(z-\gamma\overline z_0)$. It suffices to prove that
  if
  $$
  y(z)=\frac1{1-w}w^\alpha\sum_{n=0}^\infty c_nw^n
  $$
  is a solution of \eqref{(1.1)} near $z_0$, then
  $$
  \wt y(z)=\frac1{1-\wt w}\wt w^\alpha\sum_{n=0}^\infty
  c_n(C\wt w)^n, \qquad C=\frac{cz_0+d}{c\overline z_0+d},
  $$
  is a solution of \eqref{(1.1)} near $\gamma z_0$.

  Since \eqref{(1.1)} is assumed to be Fuchsian, the order of poles of
  $Q(z)$ at $z_0$ is at most $2$. We have
  \begin{equation*}
    \begin{split}
    Q(z)=(1-w)^4\sum_{n=-2}^\infty a_nw^n
    \end{split}
  \end{equation*}
  for some complex numbers $a_n$. Then by Lemma \ref{lemma: df/dw},
  $y(z)$ being a solution of \eqref{(1.1)} near $z_0$ means that
  \begin{equation} \label{equation: change}
    \begin{split}
      &\frac1{(2i\,\Im z_0)^2}\sum_{n=0}^\infty c_n(\alpha+n)(\alpha+n-1)
      w^{\alpha+n-2} \\
    &\qquad=\left(\sum_{n=-2}^\infty a_nw^n\right)
    \left(\sum_{n=0}^\infty c_nw^{\alpha+n}\right).
    \end{split}
  \end{equation}
  On the other hand, by Lemmas \ref{lemma: df/dw} and \ref{lemma:
    other points}, we have
  \begin{equation*}
    \begin{split}
    Q(z)=(cz_0+d)^4(1-\wt w)^4\sum_{n=-2}^\infty a_n(C\wt w)^n
    \end{split}
  \end{equation*}
  near $\gamma z_0$ and
  \begin{equation*}
    \begin{split}
    \wt y''(z)&=\frac{C^2(1-\wt w)^3}{(2i\,\Im\gamma z_0)^2}
    \sum_{n=0}^\infty c_n(\alpha+n)(\alpha+n-1)C^n\wt w^{\alpha+n-2} \\
    &=(cz_0+d)^4\frac{(1-\wt w)^3}{(2i\,\Im z_0)^2}
    \sum_{n=0}^\infty c_n(\alpha+n)(\alpha+n-1)C^n\wt w^{\alpha+n-2},
    \end{split}
  \end{equation*}
  where in the last step we have used \eqref{equation: Im} and
  $C=(cz_0+d)/(c\overline z_0+d)$. From these
  two expressions and \eqref{equation: change}, we see that if $y(z)$
  is a solution of \eqref{(1.1)} near $z_0$, then $\wt y(z)$
  is a solution of \eqref{equation: DE 3} near $\gamma z_0$, and the
  proof is completed.
\end{proof}

For our purpose, we need the following properties of power series
expansions of modular forms on $\SL(2,\Z)$. These properties are
well-known to experts (see \cite{I-O-elliptic}, for example). For
convenience of the reader, we reproduce the proofs here.
%Note that many properties
%mentioned here are special cases of general results about modular
%forms on Fuchsian subgroups of $\SL(2,\R)$ of the first kind. Here to
%make the exposition more accessible to non-specialists, we restrict
%our attention to the case of modular forms on $\SL(2,\Z)$.

\begin{Lemma} Let
  $$
  w_i(z)=\frac{z-i}{z+i}.
  $$
  Then
  $$
  w_i(-1/z)=-w_i(z), \qquad 1-w_i(-1/z)=-iz(1-w_i(z)).
  $$
  Also, let $\rho=(1+\sqrt{-3})/2$,
  $$
  w_\rho(z)=\frac{z-\rho}{z-\overline\rho}
  $$
  and $\gamma=\SM0{-1}1{-1}$. Then 
  $$
  w_\rho(\gamma z)=e^{2\pi i/3}w_\rho(z), \qquad
  1-w_\rho(\gamma z)=e^{4\pi i/3}(z-1)(1-w_\rho(z)).
  $$
\end{Lemma}

\begin{proof} The proof is straightforward. Here we will only provide
  details for the case of $w_\rho(z)$.

  We have
  $$
  w_\rho(z)=\M1{-\rho}1{-\overline\rho}z.
  $$
  Hence,
  $$
  w_\rho(\gamma z)=\M1{-\rho}1{-\overline\rho}\M0{-1}1{-1}z.
  $$
  We then compute that
  $$
  \M1{-\rho}1{-\overline\rho}\M0{-1}1{-1}
  \M1{-\rho}1{-\overline\rho}^{-1}
  =\M{(-1-\sqrt{-3})/2}00{(-1+\sqrt{-3})/2}.
  $$
  It follows that
  $$
  w_\rho(\gamma z)=e^{2\pi i/3}w_\rho(z).
  $$
  Then we have
  $$
  1-w_\rho(\gamma z)=1-\rho^2w_\rho(z)
  =1-\frac{\rho^2z+1}{z-\overline\rho}
  =\frac{(1-\rho^2)(z-1)}{z-\overline\rho},
  $$
  while
  $$
  1-w_\rho(z)=\frac{\rho-\rho^{-1}}{z-\overline\rho}.
  $$
  Hence,
  $$
  1-w_\rho(\gamma z)=-\rho(z-1)(1-w_\rho(z))
  =e^{4\pi i/3}(z-1)(1-w_\rho(z)).
  $$
  This proves the lemma.
\end{proof}

From the lemma, we deduce the following properties of expansions of
modular forms at $i$ and $\rho$. These properties will be crucial in
the proofs of Theorem \ref{theorem: 1.5}(a) and Theorem \ref{theorem: 1.7}(a).

\begin{Corollary}\label{corollary: expansion of f at i} Let $f(z)$ be a
  meromorphic modular form of even weight $k$ on $\SL(2,\Z)$. Suppose
  that the power series expansion of $f$ at $i$ is
  $$
  f(z)=(1-w_i(z))^k\sum_{n=n_0}^\infty a_nw_i(z)^n, \qquad
  w_i(z)=\frac{z-i}{z+i}.
  $$
  Then $a_n=0$ whenever $n+k/2\not\equiv 0\mod 2$. Also, if the power
  series expansion of $f$ at $\rho=(1+\sqrt{-3})/2$ is
  $$
  f(z)=(1-w_\rho(z))^k\sum_{n=n_0}^\infty b_nw_\rho(z)^n, \qquad
  w_\rho(z)=\frac{z-\rho}{z-\overline\rho},
  $$
  then $b_n=0$ whenever $n+k/2\not\equiv 0\mod 3$.
\end{Corollary}

\begin{proof} Here we will only prove the case of $\rho$.
  Let $\gamma=\SM0{-1}1{-1}$. Since $f(z)$
  is a meromorphic modular form of weight $k$ on $\SL(2,\Z)$,
  we have
  $$
  f(\gamma z)=(z-1)^kf(z)
  =(z-1)^k(1-w_\rho(z))^k\sum_{n=n_0}^\infty b_nw_\rho(z)^n
  $$
  On the other hand, by the lemma above, we have
  $$
  f(\gamma z)=e^{4\pi ik/3}(z-1)^k(1-w_\rho(z))^k
  \sum_{n=n_0}^\infty b_ne^{2\pi in/3}w_\rho(z)^n.
  $$
  Comparing the two expressions, we conclude that $b_n=0$ whenever
  $n+k/2\not\equiv0\mod 3$.
\end{proof}

To determine local exponents of modular differential equations at
$\rho$ and $i$, we need to know the leading terms of the expansions of
$E_6(z)^2/E_4(z)^2$ and $E_4(z)^4/E_6(z)^2$.

\begin{Lemma} \label{lemma: leading terms of E4 E6}
  \begin{enumerate}
    \item[(a)] Let
  $$
  w_\rho=w_\rho(z)=\frac{z-\rho}{z-\overline\rho}.
  $$
  Then we have
  $$
  \pi^2\frac{E_6(z)^2}{E_4(z)^2}
  =(1-w_\rho^4)\left(\frac34w_\rho^{-2}+\sum_{n=1}^\infty a_nw_\rho^n\right)
  $$
  for some complex numbers $a_n$ such that $a_n=0$ whenever
  $n\not\equiv 1\mod 3$.
\item[(b)] Let
  $$
  w_i=w_i(z)=\frac{z-i}{z+i}.
  $$
  Then
  $$
  \pi^2\frac{E_4(z)^4}{E_6(z)^2}=(1-w_i)^4
  \left(\frac14w_i^{-2}+\sum_{n=0}^\infty b_nw_i^n\right)
  $$
  for some complex numbers $b_n$ such that $a_n=0$ whenever
  $n\not\equiv0\mod 2$.
\end{enumerate}
\end{Lemma}

\begin{proof} It is known that, as an analytic function on $\H$,
  $E_4(z)$ has a simple zero at $\rho$. Also, $E_6(\rho)\neq 0$. Thus,
  by Corollary \ref{corollary: expansion of f at i},
  $$
  \pi^2\frac{E_6(z)^2}{E_4(z)^2}=(1-w_\rho)^4\left(
    a_{-2}w_\rho^{-2}+\sum_{n=1}^\infty a_nw_\rho^n\right)
  $$
  for some complex numbers $a_n$ such that $a_n=0$ whenever
  $n\not\equiv1\mod 3$. To determine the leading coefficient $a_{-2}$,
  we use the well-known Ramanujan's identity
  $$
  \frac1{2\pi i}E_4'(z)=\frac{E_2(z)E_4(z)-E_6(z)}3,
  $$
  where $E_2(z)$ is the Eisenstein series of weight $2$ on $\SL(2,\Z)$
  (see \cite[Proposition 15]{Zagier123}). Hence,
  \begin{equation*}
    \begin{split}
    \lim_{z\to\rho}w_\rho(z)\frac{E_6(z)}{E_4(z)}
    &=\frac{E_6(\rho)}{\rho-\overline\rho}\lim_{z\to\rho}
    \frac{z-\rho}{E_4(z)}=\frac{E_6(\rho)}{\sqrt3i}\frac1{E'_4(\rho)}\\
    &=-\frac{E_6(\rho)}{2\pi\sqrt3}\frac3{E_2(\rho)E_4(\rho)-E_6(\rho)}
    =\frac{\sqrt 3}{2\pi},
    \end{split}
  \end{equation*}
  which implies that $a_{-2}=3/4$. This proves Part (a).

  The proof of Part (b) is similar. We use another identity
  $$
  \frac1{2\pi i}E_6'(z)=\frac{E_2(z)E_6(z)-E_4(z)^2}2
  $$
  of Ramanujan's to conclude that the leading term of
  $\pi^2E_4(z)^4/E_6(z)^2$ is $w_i^{-2}/4$. We omit the details.
\end{proof}

%Now, we apply Corollary \ref{corollary: expansion of f at i} to the modular
%form $Q_1(z;r,s)$ at $\rho$ and obtain

\begin{Corollary} \label{corollary: indicial}
  The local exponents of the modular differential equation
  \eqref{equation: DE 3} at $\rho$ and at $i$ are roots of
  $$
  x^2-x+\frac 94s=0
  $$
  and
  $$
  x^2-x+t=0,
  $$
  respectively.
\end{Corollary}

\begin{proof} Here we prove only the case of $\rho$; the proof of the
  case of $i$ is similar.

  Let $w=w_\rho(z)=(z-\rho)/(z-\overline\rho)$. Assume that
  $$
  y(z)=\frac1{1-w}\sum_{n=0}^\infty a_nw^{\alpha+n}, \quad a_0\neq 0,
  $$
  is a solution of \eqref{equation: DE 3}. By Lemmas \ref{lemma:
    leading terms of E4 E6} and \ref{lemma: df/dw}, we have
  $$
  y''(z)=-\frac{(1-w)^3}3\left(\alpha(\alpha-1)a_0w^{\alpha-2}+\cdots\right)
%  \sum_{n=0}^\infty a_n(\alpha+n)
%  (\alpha+n-1)w^{\alpha+n-2},
  $$
  while
  \begin{equation*}
    \begin{split}
  &\pi^2\left(rE_4(z)+s\frac{E_6(z)^2}{E_4(z)^2}+t\frac{E_4(z)^4}{E_6(z)^2}
  \right)y(z) \\
  &\qquad=(1-w)^3\left(\frac34sa_0w^{\alpha-2}+\cdots\right).
   \end{split}
  \end{equation*}
  Comparing the leading terms, we see that the exponent $\alpha$
  satisfies $\alpha^2-\alpha+9s/4=0$.
\end{proof}

%Write a solution $y$ of (\ref{equation DE in introduction}) by
%  \begin{equation*} \label{equation: f in w}
%  y(z)=\frac1{1-w}w^\alpha\sum_{n=0}^\infty c_n
%  u^n\ \ \ \text{with}\ c_0=1,\ \alpha=-(1/2+\kappa_\rho).
%  \end{equation*}
%  Then by (\ref{equation: E4}), (\ref{equation: E6/E4})
%and Lemma \ref{lemma: df/dw}, $c_n$ satisfies
%  \begin{equation}\label{equation: recursive 1}
%  n\left(n-2\kappa_\rho\right)c_n=
%  -\frac{3}{4}\sum_{j=0}^{n-2}c_j
%  (ra_{n-j-2} +sb_{n-j-2}).
%  \end{equation}
%  Due to (\ref{equation: conditions for an, bn not zero}) and
%  (\ref{equation: recursive 1}), we can inductively prove
%  \begin{equation}\label{equation: condition for cn}
%  c_n=0\ \ \ \text{if}\ n\not\equiv 0\mod 3.
%  \end{equation}

We are now ready to prove Part (a) of Theorem \ref{theorem: 1.5}.

\begin{proof}[Proof of Theorem \ref{theorem: 1.5}(a)]
  By Proposition \ref{proposition: same exponents}, we only need to
  determine when \eqref{equation: DE in introduction} is apparent at
  $\rho$.
  
  Let $\kappa_\rho\in\frac12\N$ and set
  $s=s_{\kappa_\rho}=(1-4\kappa_\rho)/9$ so that the local exponents
   of the modular differential equation \eqref{equation: DE in
     introduction} with $s=s_{\kappa_\rho}$, i.e.,
   \begin{equation} \label{equation: sk}
   y''(z)=\pi^2\left(rE_4(z)+s_{\kappa_\rho}
     \frac{E_6(z)^2}{E_4(z)^2}\right)y(z)
   \end{equation}
   at $\rho$ are $1/2\pm\kappa_\rho$, by Corollary \ref{corollary: indicial}.

   Let $w=w_\rho(z)=(z-\rho)/(z-\overline\rho)$. According to
   Corollary \ref{corollary: expansion of f at i} and
   Lemma \ref{lemma: leading terms of E4 E6}, we have
  \begin{equation} \label{equation: E4}
  \pi^2E_4(z)=(1-w)^4\sum_{n=1}^\infty a_nw^n,
  \end{equation}
  and
  \begin{equation} \label{equation: E6/E4}
  \pi^2\frac{E_6(z)^2}{E_4(z)^2}
  =(1-w)^4\left(\frac34w^{-2}+
   \sum_{n=1}^\infty b_n
   w^n\right),
  \end{equation}
  where $a_n$ and $b_n$ are complex numbers satisfying
  \begin{equation}\label{equation: conditions for an, bn not zero}
  a_n=b_n=0\quad \text{if}\quad n\not\equiv 1\mod 3.
  \end{equation}
  We also remark that $a_1\neq 0$ since the zero $\rho$ of $E_4(z)$,
  as a holomorphic function on $\H$, is simple.

  Now the differential equation \eqref{equation: sk} is apparent at
  $\rho$ if and only if it has a solution of the form
  \begin{equation*} \label{equation: f in w}
  y(z)=\frac1{1-w}w^{1/2-\kappa_\rho}\sum_{n=0}^\infty c_n
  w^n\quad \text{with}\ c_0=1.
  \end{equation*}
  Plugging this series into \eqref{equation: sk} and using Lemma
  \ref{lemma: df/dw}, \eqref{equation: E4}, and \eqref{equation:
    E6/E4}, we find that the coefficients $c_n$ need to satisfy
 \begin{equation}\label{equation: recursive 1}
  n\left(n-2\kappa_\rho\right)c_n=
  -3\sum_{j=0}^{n-2}c_j(ra_{n-j-2} +s_{\kappa_\rho}b_{n-j-2}).
  \end{equation}
  Due to \eqref{equation: conditions for an, bn not zero} and
  \eqref{equation: recursive 1}, we can inductively prove that
  \begin{equation}\label{equation: condition for cn}
  c_n=0\quad\text{if}\ n\not\equiv 0\mod 3.
  \end{equation}
% Recall $s_m$ and $\hat{s}_m$ in Theorem \ref{thm 1.5}. We have
%$s_m=s_k$ and $k=m$ for the case (a) and $\hat{s}_m=s_k$ and
%$k=m+1/2$ for the case (b).
  Since the left-hand side of \eqref{equation: recursive 1} vanishes
  when $n=2\kappa_\rho$, \eqref{equation: sk} is apparent at
  $\rho$ if and only if  
  \begin{equation} \label{equation: consistency}
    \sum_{j=0}^{2\kappa_\rho-2}c_j( ra_{2\kappa_\rho-j-2}
    +  s_{\kappa_\rho}b_{2\kappa_\rho-j-2})=0.
  \end{equation}

  Suppose $3\nmid 2\kappa_\rho$. Then,
  $j\equiv 0\mod 3$ and $2\kappa_\rho-j-2\equiv 1\mod 3$ cannot hold
  simultaneously. Hence, by \eqref{equation: conditions for an, bn not
    zero} and \eqref{equation: condition for cn}, the condition
  \eqref{equation: consistency} always holds for any $r$, i.e.,
  \eqref{equation: sk} is apparent at $\rho$ for any $r$. This proves (a).

  For the case $3|2\kappa_\rho$, considering $r$ as an indeterminate
  and using \eqref{equation: recursive 1} to recursively express $c_n$ as
  polynomials in $r$, we find that $c_n$ is a polynomial of
  degree exactly $n/3$ in $r$ when $3|n$ and $n<2\kappa_\rho$.
  (Note that we use the fact that $a_1\neq 0$ to conclude that the
  degree is $n/3$.) Thus, the left-hand side of \eqref{equation:
    consistency} is a polynomial $P(r)$ of degree $2\kappa_\rho/3$ in
  $r$ and \eqref{equation: sk} is apparent at $\rho$ if and only if
  $r$ is a root of this polynomial $P(x)$. This proves Part (b)
  except the identity \eqref{equation: r+s=-(l+1/2)^2}.
% \begin{Lemma}\label{lem 3.4}
%If $2\kappa_\rho/3\in\N$, then there exists a polynomial $P$ of
%degree $2\kappa_\rho/3$ such that (\ref{equation: DE in introduction})
%with $(r,s_m)$ is apparent if and only if $r$ is a root of $P$.
% \end{Lemma}
% \begin{proof}
% Here we use the same notations as Lemma \ref{lemma: df/dw}, and want
% to prove (\ref{equation: consistency}) at $n=2\kappa_\rho$. Recall
% $c_m=0$ if $m\not\equiv 0\mod 3$. Then the recursive formula
% (\ref{equation: recursive 1}) implies $c_m$ is a polynomial of degree
% $m/3$ in $r$ if $3|m$. Thus the LHS of (\ref{equation:
% consistency}), denoted by $P(r)$, is of degree
% $2\kappa_\rho/3$. Then (\ref{equation: DE in introduction}) is
% apparent if and only if $P(r)=0$.
%  This proves the lemma. 
 \end{proof}

The proof of Theorem \ref{theorem: 1.7}(a) except \eqref{equation: r+t=-(l+1/3)^2} is very similar to that of
Theorem \ref{theorem: 1.5} and will be omitted.

\section{ Riemann's existence theorem and its application. }
In this section, we will use Riemann's existence theorem to prove
Theorems \ref{theorem: necessary and sufficient conditions for appartness at infinity}, \ref{theorem: 1.5}(b), and \ref{theorem: 1.7}(b). The
basic idea is as follows.

Let $h(z)$ be a modular function on some subgroup $\Gamma$ of finite
index of $\SL(2,\Z)$. A simple computation
shows that both $y_1(z)=1/\sqrt{h'(z)}$ and $y_2(z)=h(z)/\sqrt{h'(z)}$
are solutions of
$$
y''(z)=Q(z)y(z), \qquad Q(z)=-\frac12\{h(z),z\},
$$
where $\{h(z),z\}$ is the Schwarz derivative. Using either properties
of Schwarz derivatives or direct computation, we can verify that
$\{h(z),z\}$ is a meromorphic modular form of weight $4$ on $\Gamma$.
When $h(z)$ has additional symmetry, $\{h(z),z\}$ can be modular on
a larger group. Note that, by construction, this differential equation
$y''(z)=Q(z)y(z)$ is apparent on $\H$. Thus, one way to prove the
theorems is simply to prove the existence of a modular function
$h(z)$ such that $-\{h(z),z\}/2=Q(z)$ for each $Q(z)$ appearing in the
theorems. To achieve this, we will use Riemann's existence theorem.

Since some of the readers may not be familiar with Riemann's existence
theorem, here we give a quick overview of this important result in the
theory of Riemann surfaces. The exposition follows \cite[Chapter
III]{Miranda}.

Let $F:X\to Y$ be a (branched) covering of compact Riemann surfaces of
degree $d$. A point $y$ of $Y$ is a \emph{branch point} if the
cardinality of $F^{-1}(y)$ is not $d$ and a point $x$ of $X$ is a
\emph{ramification point} if $F$ is not locally one-to-one near $x$.
(In particular, $F(x)$ is a branch point.)
Let $B$ be the (finite) set of branch points on $Y$ under $F$. Pick a
point $y_0\in Y-B$ so that $F^{-1}(y_0)$ has $d$ points, say
$x_1,\ldots,x_d$. Every loop $\gamma$ in $Y-B$ based at $y_0$ can be
lifted to $d$ paths $\widetilde\gamma_1,\ldots,\widetilde\gamma_d$
with $\widetilde\gamma_j(0)=x_j$ and $\widetilde\gamma_j(1)=x_{j'}$
for some $x_{j'}$. The map $j\mapsto j'$ is then a permutation in
$S_d$. The permutation depends only on the homotopy class of
$\gamma$. In this way, we get a monodromy representation
$$
\rho:\pi_1(Y-B,y_0)\to S_d.
$$
Note that since $F^{-1}(Y-B)$ is connected, the image of $\rho$ is a
transitive subgroup of $S_d$. Also, let $b\in B$ and $a_1,\ldots,a_k$
be the points in $F^{-1}(b)$ with ramification indices
$m_1,\ldots,m_k$, respectively. We can show that if $\gamma$ is a
small loop in $Y-B$ around $b$ based at $y_0$, then $\rho(\gamma)$ is
a product of disjoint cycles of lengths $m_1,\ldots,m_k$.

To state the version of Riemann's existence theorem used in the paper,
let us consider the case $Y=\P^1(\C)$. Let $B=\{b_1,\ldots,b_n\}$ be the 
set of branch points of $F:X\to\P^1(\C)$. Let $\gamma_j$,
$j=1,\ldots,n$, be loops that circles $b_j$ once but no other branch points.
Then $\pi_1(\P^1(\C)-B,y_0)$ is generated by the homotopy classes
$[\gamma_j]$, subject to a single relation
$[\gamma_1]\ldots[\gamma_n]=1$ (with a suitable ordering of the points
$b_j$). Thus, the image of $\rho$ is generated by
$\sigma_j=\rho(\gamma_j)$ satisfying the relation
$\sigma_1\cdots\sigma_n=1$. Then Riemann's existence theorem states
as follows (see \cite[Corollary 4.10]{Miranda}).

\begin{theorem}[Riemann's existence theorem]
  Let $B=\{b_1,\ldots,b_n\}$ be a finite subset of $\P^1(\C)$. Then
  there exists a one-to-one correspondence between the set of
  isomorphism classes of coverings $F:X\to\P^1(\C)$ of compact Riemann
  surfaces of degree $d$ whose branch points lie in $B$ and the set of
  (simultaneous) conjugacy classes of $n$-tuples
  $(\sigma_1,\ldots,\sigma_n)$ of permutations in $S_d$ such that
  $\sigma_1\ldots\sigma_n=1$ and the group generated by the
  $\sigma_j$'s is transitive.

  Moreover, if the disjoint cycle decomposition of $\sigma_j$ is a
  product of $k$ cycles of lengths $m_1,\ldots,m_k$, then
  $F^{-1}(b_j)$ has $k$ points with ramification indices
  $m_1,\ldots,m_k$, respectively.
\end{theorem}

We now use this result to prove Theorems \ref{theorem: necessary and sufficient conditions for appartness at infinity}, \ref{theorem:
  1.5}(b), and \ref{theorem: 1.7}(b). Since the proofs are similar, we will
provide details only for Theorem \ref{theorem: 1.5}(b).

\begin{proof}[Proof of Theorem \ref{theorem: 1.5}(b)]
  Assume that $3|2\kappa_\rho$. Let $\Gamma_2$ be the subgroup of
  index of $2$ of $\SL(2,\Z)$ generated by
  $$
  \gamma_1=\M1{-1}10, \qquad\gamma_2=\M01{-1}{-1}.
  $$
  Note that
  $$
  \gamma_1\gamma_2=\M1201.
  $$
  The group $\Gamma_2$ has a cusp $\infty$ and two elliptic points
  $\rho_1=(1+\sqrt{-3})/2$ and $\rho_2=(-1+\sqrt{-3})/2$ of order $3$,
  fixed by $\gamma_1$ and $\gamma_2$, respectively. Let
  $$
  j_2(z)=\frac{E_6(z)}{\eta(z)^{12}},
  $$
  which is a Hauptmodul for $\Gamma_2$, and set
  $$
  J_2(z)=\frac{24}{j_2(z)}.
  $$
  We have $J_2(\infty)=0$, $J_2(\rho_1)=1/\sqrt{-3}$, and
  $J_2(\rho_2)=-1/\sqrt{-3}$.

  Set $\ell_0=2\kappa_\rho/3$. We first show that for each
  $\ell\in\{0,\ldots,\ell_0-1\}$, there exists a modular function $h(z)$
  on $\Gamma_2$ such that the covering $h:X(\Gamma_2)\to\mathbb P^1(\C)$
  of compact Riemann surfaces is ramified precisely 
  at $\infty$, $\rho_1$, and $\rho_2$ with 
  ramification index $2\ell+1$, $\ell_0$, and $\ell_0$, respectively.
  Note that by the Riemann-Hurwitz formula, such a covering has degree
  $\ell_0+\ell$, i.e., such a modular function $h(z)$ will be a rational
  function of degree $\ell_0+\ell$ in $J_2(z)$.

  Consider the two $\ell_0$-cycles
  $$
  \sigma_1=(1,\ldots,\ell_0), \qquad
  \sigma_2=(\ell_0+\ell,\ell_0+\ell-1,\ldots,\ell+1)
  $$
  in the symmetric group $S_{\ell_0+\ell}$. Since $\ell<\ell_0$, we
  have
  $$
  \sigma_2\sigma_1=(1,\ldots,\ell,\ell_0+\ell,\ell_0+\ell-1,
  \ldots,\ell_0),
  $$
  which is a $(2\ell+1)$-cycle. (Notice that if $\ell\ge\ell_0$, then
  $\sigma_1$ and $\sigma_2$ are disjoint.) It is clear that when
  $\ell<\ell_0$, the subgroup generated by $\sigma_1$ and $\sigma_2$
  is a transitive subgroup of $S_{\ell_0+\ell}$. Thus, by Riemann's
  existence theorem, there exists a covering of compact Riemann
  surfaces $H:X\to\mathbb P^1(\C)$ of degree $\ell_0+\ell$ ramified at
  three points $\zeta_1$, $\zeta_2$, and $\zeta_3$ of $\mathbb P^1(\C)$ with
  corresponding monodromy $\sigma_1$,
  $\sigma_2$, and $\sigma_1^{-1}\sigma_2^{-1}$, respectively. By the
  Riemann-Hurwitz formula, the genus of $X$ is $0$, and $H$ is a
  rational function from $\mathbb P^1(\C)$ to $\mathbb
  P^1(\C)$. Furthermore, by applying a suitable linear fractional
  transformation on the variable
  of $H$, we may assume that the three ramified points in
  $H^{-1}(z_j)$ are $0=J_2(\infty)$, $1/\sqrt{-3}=J_2(\rho_1)$, and
  $-1/\sqrt{-3}=J_2(\rho_2)$, respectively. Set $h(z)=H(J_2(z))$. Then
  $h(z)$ has the required properties that the only points of
  $X(\Gamma_2)$ ramified under $h:X(\Gamma_2)\to\mathbb P^1(\C)$ are
  $\rho_1$, $\rho_2$, and the cusp $\infty$ with ramified indices
  $\ell_0$, $\ell_0$, and $2\ell+1$, respectively.

  Now consider the Schwarz derivative $\{h(z),z\}$, which is a
  meromorphic modular form of weight $4$ on $\Gamma_2$. We claim
  that it is in fact modular on the bigger group $\SL(2,\Z)$.

  Indeed, to show $\{h(z),z\}$ is modular on $\SL(2,\Z)$, it suffices
  to prove that $\{h(z),z\}\big|T=\{h(z),z\}$, where $T=\SM1101$. Let
  $\widetilde h(z)=h(z+1)$. Now the automorphism on $X(\Gamma_2)$
  induced by $T$ interchanges $\rho_1$ and $\rho_2$. Thus, the
  ramification data of the covering
  $\widetilde h:X(\Gamma_2)\to\mathbb P^1(\C)$ is the same as that of $h$. By
  the Riemann's existence theorem, $h$ and $\wt h$ are related by a
  linear fractional transformation, i.e.,
  $\widetilde h=(ah+b)/(ch+d)$ for some $a,b,c,d\in\C$ with $ad-bc\neq
  0$. It follows that $\{h(z),z\}\big|T=\{h(z),z\}$ by the well-known
  property $\{(af(z)+b)/(cf(z)+d),z\}=\{f(z),z\}$ of the Schwarz
  derivative. This proves that $\{h(z),z\}$ is a meromorphic modular
  form of weight $4$ on the larger group $\SL(2,\Z)$.

  Furthermore, since $\rho_1$ is an elliptic point of order $3$, a local
  parameter for $\rho_1$ as a point on the compact Riemann surface
  $X(\Gamma_2)$ is $w^3$, where $w=(z-\rho)/(z-\overline\rho)$.
  Therefore, we have
  $$
  h(z)=d_0+\sum_{n=3\ell_0}^\infty d_nw^n,
  $$
  for some complex numbers $d_n$ with $d_{3\ell_0}\neq0$ and $d_n=0$
  whenever $3\nmid n$. For convenience, set
  \begin{equation*}
    \begin{split}
      A&=\sum_{n=3\ell_0}^\infty nd_nw^{n-1}, \\
      B&=\sum_{n=3\ell_0}^\infty n(n-1)d_nw^{n-2}, \\
      C&=\sum_{n=3\ell_0}^\infty n(n-1)(n-2)d_nw^{n-3}.
    \end{split}
  \end{equation*}
  Using \eqref{equation: w'}, we compute that
  \begin{equation*}
    \begin{split}
      h'(z)&=\frac{(1-w)^2}{\rho-\overline\rho}A, \\
      h''(z)&=\frac{(1-w)^4}{(\rho-\overline\rho)^2}B
      -2\frac{(1-w)^3}{(\rho-\overline\rho)^2}A, \\
      h'''(z)&=\frac{(1-w)^6}{(\rho-\overline\rho)^3}C
      -6\frac{(1-w)^5}{(\rho-\overline\rho)^3}B
      +6\frac{(1-w)^4}{(\rho-\overline\rho)^3}A,
    \end{split}
  \end{equation*}
  and hence
  $$
  \{h(z),z\}=\frac{(1-w)^4}{(\rho-\overline\rho)^2}\left(\frac CA
    -\frac32\frac{B^2}{A^2}\right)
  =-\frac{(1-w)^4}3\left(\frac{1-9\ell_0^2}{2w^2}+cw+\cdots\right)
  $$
  for some $c$. It follows that, by \eqref{equation: E6/E4},
  $$
  \{h(z),z\}+2\pi^2s_{\kappa_\rho}\frac{E_6(z)^2}{E_4(z)^2}, \qquad
  s_{\kappa_\rho}=\frac{1-4\kappa_\rho^2}9=\frac19-\ell_0^2,
  $$
  is a holomorphic modular form of weight $4$ on
  $\mathrm{SL}(2,\Z)$. By comparing the leading coefficients of the
  Fourier expansions at the cusp $\infty$, we conclude that,
  $$
  \{h(z),z\}=-2\pi^2\left(rE_4(z)+s_{\kappa_\rho}
    \frac{E_6(z)^2}{E_4(z)^2}\right),
  $$
  where $r=-(2\ell+1)^2/4-s_{\kappa_\rho}=\ell_0^2-(2\ell+1)^2/4-1/9$.
  Equivalently, $1/\sqrt{h'(z)}$ and $h(z)/\sqrt{h'(z)}$ are solutions
  of \eqref{equation: DE in introduction} which also implies that the
  singularity of \eqref{equation: DE in introduction} at $\rho$ is apparent.

  Finally, since we have found $\ell_0$ different $r$ such that
  \eqref{equation: DE in introduction} has an apparent singularity at
  $\rho$ for the given $s_{\kappa_\rho}$, by Part (a), this proves the theorem.
 \end{proof}

\begin{Example} For small $\kappa_\rho$, the modular functions $h(z)$
  appearing in the proof are given by 
   $$ \extrarowheight3pt
  \begin{array}{cccc} \hline\hline
    \kappa_\rho & \ell & (r,s) & h(z) \\ \hline
      \displaystyle\frac32 & 0 & \displaystyle
      \left(\frac{23}{36},-\frac89\right)\phantom{\Bigg|}
                               & J_2 \\
    3 & 0 & \displaystyle
    \left(\frac{131}{36},-\frac{35}9\right) \phantom{\Bigg|}
    & \displaystyle\frac{J_2}{1-3J_2^2} \\
    3 & 1 & \displaystyle
    \left(\frac{59}{36},-\frac{35}9\right) \phantom{\Bigg|}
    & \displaystyle\frac{J_2^3}{1+9J_2^2} \\
            \hline\hline
  \end{array}
  $$
\end{Example}

\begin{proof}[Proof of Theorem \ref{theorem: 1.7}(b)]
  Assume that $\kappa_i\in\N$. Let $\Gamma_3$ be the subgroup of index
  $3$ of $\mathrm{SL}(2,\Z)$ generated by
  $$
  \gamma_1=\M1{-2}1{-1}, \qquad \gamma_2=\M1{-1}2{-1}, \qquad
  \gamma_3=\M0{-1}10.
  $$
  We note that
  $$
  \gamma_1\gamma_2\gamma_3=\M1301.
  $$
  The group $\Gamma_3$ has one cusp and three elliptic points
  $z_1=1+i$, $z_2=(1+i)/2$, and $z_3=i$ of order $2$, fixed by
  $\gamma_j$, $j=1,2,3$, respectively.
  Let
  $$
  j_3(z)=\frac{E_4(z)}{\eta(z)^8}
%  =\left(\frac{\eta(z)^{12}}{\eta(z/3)^6\eta(3z)^6}-6
%    -27\frac{\eta(z/3)^6\eta(3z)^6}{\eta(z)^{12}}\right)
  $$
  be a Hauptmodul for $\Gamma_3$ and set
  $$
  J_3(z)=12j_3(z)^{-1}.
  $$
  Note that $j_3(z)^3$ is equal to the elliptic $j$-function
  $j(z)$. Since $j(i)=1728$ and $j(\rho)=0$,
  we have $\{J_3(z_1),J_3(z_2),J_3(z_3)\}=\{1,e^{2\pi i/3},e^{4\pi
    i/3}\}$, $J_3(\rho)=\infty$, and $J_3(\infty)=0$.

  Consider the case $r+t_{\kappa_i}=-(\ell+1/3)^2$ first. Our goal
  here is to construct a modular function $h(z)$ on $\Gamma_3$, for
  each $\ell$ in the range, such that the covering
  $h:X(\Gamma_3)\to\P^1(\C)$ has degree
  $$
  d=\frac12(3\kappa_i+3\ell-1)
  $$
  and is ramified at precisely the cusp $\infty$ and the three
  elliptic points $z_1$, $z_2$, and $z_3$ with ramification indices
  $3\ell+1$, $\kappa_i$, $\kappa_i$, and $\kappa_i$, respectively.
  (Notice that $\kappa_i$ and $\ell$ have opposite parities, so $d$ is an
  integer.) Since the covering has four branch points, it is not easy
  to apply Riemann's existence theorem directly to get $h(z)$. Instead, we
  shall use the following idea.

  For convenience, set
  \begin{equation} \label{equation: m m'}
  m=\frac12(\kappa_i+\ell-1), \qquad
  m'=\frac12(\kappa_i-\ell-1).
  \end{equation}
  We claim that there exists a rational function $H(x)$ of degree $d$ in
  $x$ of the form
  $$
  H(x)=\frac{x^{3\ell+1}G(x)^3}{F(x)^3}, \quad
  \deg F(x)=m, \quad \deg G(x)=m',
  $$
  such that $xF(x)G(x)$ is squarefree and
  $$
  H(x)-1=\frac{(x-1)^{\kappa_i}L(x)}{F(x)^3}
  $$
  for some polynomial $L$ of degree $d-\kappa_i$ with no repeated
  roots. That is, $H(x)$ is a rational function such that
  \begin{enumerate}
    \item[(i)] the covering $H:\P^1(\C)\to\P^1(\C)$ branches at
      precisely $\infty$, $0$, and $1$ (note that by the
      Riemann-Hurwitz formula, $H$ cannot have other branch points),
    \item[(ii)] the monodromy $\sigma_\infty$ around $\infty$ is a
      product of $m$ disjoint $3$-cycles, the
      monodromy $\sigma_0$ around $0$ is a disjoint product of a
      $(3\ell+1)$-cycle and $m'$ $3$-cycles, and the monodromy
      $\sigma_1$ around $1$ is a $\kappa_i$-cycle,
    \item[(iii)] the unique unramified point in $H^{-1}(\infty)$ is
      $\infty$, the unique point of ramification index
      $3\ell+1$ in $H^{-1}(0)$ is $0$, and the unique ramified point
      in $H^{-1}(1)$ is $1$.
  \end{enumerate}
  Suppose that such a rational function $H(x)$ exists. We define
  $h:X(\Gamma_3)\to\P^1(\C)$ by
  $$
  h(z)=H(J_3(z)^3)^{1/3}=\frac{J_3(z)^{3\ell+1}G(J_3(z)^3)}{F(J_3(z)^3)}.
  $$
  From the construction, we see that $h$ ramifies only at $z_1=1+i$,
  $z_2=(1+i)/2$, $z_3=i$, and $\infty$ with ramification indices
  $\kappa_i$, $\kappa_i$, $\kappa_i$, and $3\ell+1$, respectively.
  Then following the proof of Theorem \ref{theorem: 1.5}(b), we can prove that
  the Schwarz derivative $\{h(z),z\}$ is a meromorphic modular form on
  the larger group $\SL(2,\Z)$ and that
  $$
  \{h(z),z\}=-2\pi^2\left(rE_4(z)+t_{\kappa_i}\frac{E_4(z)^4}{E_6(z)^2}
    \right), \quad r=-\left(\ell+\frac13\right)^2-t_{\kappa_i},
  $$
  which is equivalent to the assertion that $1/\sqrt{h'(z)}$ and
  $h(z)/\sqrt{h'(z)}$ are solutions of \eqref{equation: DE 2 in
    introduction} with $t=t_{\kappa_i}$ and
  $r=-(\ell+1/3)^2-t_{\kappa_i}$ and hence implies that
  \eqref{equation: DE 2 in introduction} is apparent with these 
  parameters.
%  It follows that $r=-(\ell+1/3)^2-t_{\kappa_i}$ is a root of $P(x)$.

  It remains to prove that a rational function $H(x)$ with properties
  described above exists. According to Riemann's existence theorem, it
  suffices to find $\sigma_\infty$ that is a product of $m$ disjoint
  $3$-cycles and $\sigma_1$ that is a $\kappa_i$-cycle in $S_d$ such
  that $\sigma_1\sigma_\infty$ is a disjoint product of a cycle of
  length $3\ell+1$ and $m'$ cycles of length $3$. Indeed, we find that
  we may choose
  $$
  \sigma_\infty=(2,3,4)(5,6,7)\ldots(3m-1,3m,3m+1)
  $$
  and
  $$
  \sigma_1=(1,2,5,8,\ldots,3m-1,3m'+1,3m'-2,\ldots,7,4).
  $$
  Then
  $$
  \sigma_1\sigma_\infty=(1,2,3)(4,5,6)\ldots(3m'-2,3m'-1,3m')
  (3m'+1,3m'+2,\ldots,d).
  $$
  This settles the case $r+t_{\kappa_i}=-(\ell+1/3)^2$.

  The case $r+t_{\kappa_i}=-(\ell-1/3)^2$ can be dealt with in the
  same way. The difference is that the rational function $H(x)$
  in this case has degree
  $$
  d=\frac32(\kappa_i+\ell-1)
  $$
  and is of the form
  $$
  H(x)=\frac{x^{3\ell-1}G(x)^3}{F(x)^3}, \quad
  \deg F(x)=m, \quad \deg G(x)=m',
  $$
  where $m$ and $m'$ are the same as those in \eqref{equation: m m'},
  such that $xF(x)G(x)$ is squarefree and
  $$
  H(x)-1=\frac{(x-1)^{\kappa_i}L(x)}{F(x)^3}
  $$
  for some polynomial $L(x)$ of degree $d-\kappa_i$ with no repeated
  roots. I.e., $\sigma_\infty$ in this case is a disjoint product of
  $m$ $3$-cycles, $\sigma_0$ is a a disjoint product of
  $(3\ell-1)$-cycle and $m'$ $3$-cycles, and $\sigma_1$ is a
  $\kappa_i$-cycle. We choose
  $$
  \sigma_\infty=(1,2,3)(4,5,6)\ldots(3m-2,3m-1,3m)
  $$
  and
  $$
  \sigma_1=(1,4,7,\ldots,3m-2,3m,3m-3,\ldots,3\ell)
  $$
  with
  $$
  \sigma_1\sigma_\infty=(1,2,3,4,\ldots,3\ell-1)
  (3\ell,3\ell+1,3\ell+2)\ldots(3m-3,3m-2,3m-1).
  $$
  The rest of proof is the same as the case of
  $r+t_{\kappa_i}=-(\ell+1/3)^2$. This completes the proof that
  \eqref{equation: r+t=-(l+1/3)^2} is the complete list of parameters $r$
  such that \eqref{equation: DE 2 in introduction} with
  $t=t_{\kappa_i}$ is apparent.
\end{proof}

\begin{Example} For small $\kappa_i$, the modular functions $h(z)$ in
  the proof are given by
  $$ \extrarowheight3pt
  \begin{array}{cccc} \hline\hline
    \kappa_i & \ell\pm1/3 & (r,t) & h(z) \\ \hline
    1 & \displaystyle\frac13 & \displaystyle
    \left(\frac{23}{36},-\frac34\right) \phantom{\Bigg|}
    & J_3 \\
    2 & \displaystyle\frac23 & \displaystyle
    \left(\frac{119}{36},-\frac{15}4\right)\phantom{\Bigg|}
    &\displaystyle\frac{J_3^2}{1+2J_3^3} \\
    2 & \displaystyle\frac43 & \displaystyle
    \left(\frac{71}{36},-\frac{15}4\right)\phantom{\Bigg|}
    &\displaystyle\frac{J_3^4}{1-4J_3^3} \\
                               \hline\hline
  \end{array}
  $$
\end{Example}

\begin{proof}[Proof of Theorem \ref{theorem: necessary and sufficient conditions for appartness at infinity}]
Assume that $n_i$, $n_\rho$, and $n_\infty$ are positive
  integers satisfying the two conditions. We note that the
  parameters $r$, $s$, and $t$ in \eqref{equation: DE 3} are
  \begin{equation} \label{equation: apparent r s t}
    r=-n_\infty^2+n_\rho^2+n_i^2-\frac{13}{36}, \qquad
     s=\frac19-n_\rho^2, \qquad
     t=\frac14-n_i^2.
  \end{equation}
  Let
  $$
  d=\frac12(n_i+n_\rho+n_\infty-1).
  $$
  By the second condition, we have
  $$
  d-n_i=\frac12(n_\rho+n_\infty-n_i-1)\ge 0
  $$
  and similarly, $d-n_\rho\ge 0$. Thus, there are cycles of lengths
  $n_i$ and $n_\rho$ in the symmetric group $S_d$. Choose
  $$
  \sigma_1=(1,\ldots,n_i), \qquad
  \sigma_2=(d,d-1,\ldots,d-n_\rho+1)
  $$
  By the second condition again, we have
  $$
  n_i-(d-n_\rho+1)=\frac12(n_i+n_\rho-n_\infty-1)\ge 0.
  $$
  In other words, the two cycles are not disjoint. We then compute
  that
  $$
  \sigma_2\sigma_1=(1,\ldots,d-n_\rho,d,d-1,\ldots,n_i).
  $$
  This is a cycle of length
  $$
  d-n_\rho+(d-n_i+1)=2d-n_\rho-n_i+1=n_\infty.
  $$
  It is clear that the subgroup of $S_d$ generated by $\sigma_1$ and
  $\sigma_2$ is transitive. Thus, by Riemann's existence theorem,
  given three distinct points $\zeta_1$, $\zeta_2$, and $\zeta_3$ on
  $\P^1(\C)$, there is a covering $H:X\to\P^1(\C)$ of compact Riemann
  surfaces of degree $d$ branched at $\zeta_1$, $\zeta_2$, and
  $\zeta_3$ with monodromy $\sigma_1$, $\sigma_2$, and
  $\sigma_3=\sigma_1^{-1}\sigma_2^{-1}$,
  respectively. By the Riemann-Hurwitz formula, the genus of $X$ is
  $0$ and we may assume that $X=\P^1(\C)$. Applying a suitable
  linear fractional transformation (i.e., an automorphism of $X$) if
  necessary, we may assume that the ramification points on $X$ are
  $1728=j(i)$, $0=j(\rho)$, and $\infty=j(\infty)$ with
  ramification indices $n_i$, $n_\rho$, and $n_\infty$,
  respectively. Let $h:X_0(1)\to\P^1(\C)$ be defined by
  $h(z)=H(j(z))$. Following the same computation as in the proof of
  Theorem \ref{theorem: 1.5}(b), we can show that
%  We will show that $1/\sqrt{h'(z)}$ and $h(z)/\sqrt{h'(z)}$ are
%  solutions of (\ref{equation: DE 3}) for parameters $r,s,t$
%  given by \eqref{equation: apparent r s t}. Equivalently, we will
%  prove that
  $$
  \{h(z),z\}=-2\pi^2\left(rE_4(z)+s\frac{E_6(z)^2}{E_4(z)^2}
    +t\frac{E_4(z)^4}{E_6(z)^2}\right)
  $$
  with $r$, $s$, and $t$ given as \eqref{equation: apparent r s t}
  (details omitted). This implies that the singularities of
  \eqref{equation: DE 3} are all apparent.

  Conversely, assume that the differential equation \eqref{equation: DE
    3} is apparent throughout $\H\cup\{\text{cusps}\}$. Let 
  $\pm n_\infty/2$ be the local exponents at $\infty$. Then
  a fundamental pair of solutions near $\infty$ is
  $$
  y_\pm(z)=q^{\pm n_\infty/2}\left(1+\sum_{n=1}^\infty
    c_n^\pm q^n\right).
  $$
  Let $h(z)=y_+(z)/y_-(z)$. Since \eqref{equation: DE 3} is apparent
  throughout $\H$, $h(z)$ is a single-valued function on $\H$. Arguing
  as in the second proof of  Theorem \ref{theorem: necessary and
    sufficient conditions for appartness at infinity}, we see that
  $h(z)$ is a modular function on $\SL(2,\Z)$. Now since
  $$
  \{h(z),z\}=-2\pi^2\left(rE_4(z)+s\frac{E_6(z)^2}{E_4(z)^2}
    +t\frac{E_4(z)^4}{E_6(z)^2}\right)
  $$
  have poles only at points equivalent to $\rho$ or $i$ under
  $\SL(2,\Z)$, the covering $X_0(1)\to\P^1(\C)$ defined by $z\mapsto
  h(z)$ can only ramify at $\rho$, $i$, or $\infty$. From the
  computation above, we see that their ramification indices must be
  $n_\rho$, $n_i$, and $n_\infty$, respectively. Then by the
  Riemann-Hurwitz formula, $n_\rho+n_i+n_\infty$ must be odd and the
  degree of the covering is $(n_\rho+n_i+n_\infty-1)/2$. Since
  the ramification indices $n_\rho$, $n_i$, and $n_\infty$ cannot
  exceed the degree of the covering, we conclude that the sum of any
  two of $n_\rho$, $n_i$, and $n_\infty$ must be greater than the
  remaining one. This completes the proof of the theorem.
\end{proof}

 \section{Eremenko's Theorem and its applications}
 
\begin{proof}[Second proof of \eqref{equation: r+s=-(l+1/2)^2}]
  In Section 2.3, Example 2 shows that
the angle of $Q_1$ at $i,\rho$ and $\infty$ are
\begin{equation}
\theta_1=\frac{1}{2},\quad \theta_2=\frac{2\kappa_\rho}{3},\quad \text{and}\quad\theta_\infty=\sqrt{-(r+s_{\kappa_\rho})}.
\end{equation}

First, we consider $\theta_2$ is even, say $\theta_2=2\ell_0$.
By Eremanko's Theorem in Section 2, the curvature equation
\eqref{2.4} has a solution if and only if either
$\abs{\theta_\infty-\theta_1}=2\ell+1$ or
$\theta_\infty+\theta_1=2\ell+1$ for some $\ell\in\Z_{\geq 0}$ and
$\ell\leq \ell_0-1$.
Since $\theta_\infty>0$, the condition
$\abs{\theta_\infty-\theta_1}=2\ell+1\geq 1$ implies
$\theta_\infty-\theta_1>0$ and then
$\theta_\infty-\theta_1=2\ell+1$. This is equivalent to
$-(r+s_{\kappa_\rho})=\theta^2_\infty=(2\ell+1+1/2)^2$,
$\ell=0,\ldots,\ell_0-1$. The second condition
$\theta_\infty+\theta_1=2\ell+1$ is equivalent to
$-(r+s_{\kappa_\rho})=\theta^2_\infty=(2\ell+1/2)^2$,
$\ell=0,\ldots,\ell_0-1$.
Therefore, there are exactly $2\ell_0$ different $\theta_\infty$ such
that the curvature equation \eqref{2.4} has
a solution and each of such a curvature equation is associated with the
modular form $Q_1(z;r,s_{\kappa_\rho})$ with $(r,s_{\kappa_\rho})$ where $r+s_{\kappa_\rho}=-(\ell+1/2)^2$ for
some $\ell\in\braces{0,\ldots,2\ell_0-1}$.
By Theorem \ref{theorem: 2.1},
for each $(r,s_{\kappa_\rho})$, the ODE \eqref{equation: DE in introduction}
is apparent. However, the first part of Theorem \ref{theorem: 1.5}(b) says that
there exists a polynomial $P(x)$ of degree $2\kappa_\rho/3$ such that
\eqref{2.4} with $(r,s_{\kappa_\rho})$ is apparent if and only if
$P(r)=0$. Therefore, $P(r)$ has distinct roots and each root satisfies
$r+s_{\kappa_\rho}=-(\ell+1/2)^2$ for some integer $\ell$, $0\leq\ell\leq 2\ell_0-1=\theta_2-1$.
The proves \eqref{equation: r+s=-(l+1/2)^2} 
when $\theta_2$ is even. 

For the case $\theta_2$ is odd,
the idea of the proof is basically the same. 
By noting $\theta_1=1/2$,
the Eremenko theorem in Section 2
implies either
  $\abs{\theta_\infty-1/2}=\ell$ or $\theta_\infty+1/2=\ell$, where
  $\ell$ is even because $\theta_2$ is odd.
  The first condition can be replaced by
  $\theta_\infty-1/2=\ell$. Thus we have $\theta_\infty=\ell+1/2$ or
  $\theta_\infty=\ell-1/2=(\ell-1)+1/2$, that is
 $r+s=-(\ell+1/2)^2$, $\ell=0,1,2,\ldots,\theta_2-1$. 
The proof of \eqref{equation: r+s=-(l+1/2)^2} is complete.
\end{proof}

\begin{proof}[Second proof of \eqref{equation: r+t=-(l+1/3)^2}]
  The angles for $Q_2(z)$ are $\theta_1=\kappa_i$, $\theta_2=1/3$, and
  $\theta_\infty=\sqrt{-(r+t_i)}$, where $\frac{1}{2}\pm \kappa_i$ are
  the local exponents of \eqref{equation: DE 2 in introduction}. Hence
$$
\kappa_i-\frac{1}{2}
+1=m+\frac{1}{2}
$$
i.e., $\theta_1=\kappa_i$ is an integer. Hence, there is a solution $u$ of 
\eqref{2.4}-\eqref{2.6} with the RHS equals to $4\pi n\sum\delta_p$,
where the summation runs over $\gamma\cdot i$, $\gamma\in\SL(2,\Z)$,
if and only if either
$\theta_\infty-\theta_2=\abs{\theta_\infty-\theta_2}=\ell$ or
$\theta_\infty+\theta_2=\ell$ where $\ell\leq \kappa_i-1$ and $\ell$
has the opposite parity of $\kappa_i$. Hence,
$\theta_\infty=\ell\pm1/3$ and $r+t_i=-(\ell\pm1/3)^2$.
%Hence, $\theta_\infty=\ell+1/3$ or $\theta_\infty=(\ell-1)+1/3$. This
%implies% $-\sqrt{r+t_m}=\ell+1/3$, $\ell=0,1,\cdots,\kappa_i-1$, no
%matter $\kappa_i$ is odd or even.
This proves \eqref{equation: r+t=-(l+1/3)^2}.
\end{proof}

\begin{proof}[Second proof of Theorem \ref{theorem: necessary and
    sufficient conditions for appartness at infinity}]
Suppose that the ODE \eqref{equation: DE 3} has
local exponents $\pm n_\infty$ at $\infty$,
$n_\infty\in\frac{1}{2}\N$. We claim that \emph{\eqref{equation: DE 3}
  is apparent throughout $\H^\ast$
  if and only if $Q_3(z)=Q_3(z;r,s,t)$ is realized by a metric with
  curvature $1/2$}. It is clear that the second statement implies the
first statement. So it suffices to prove the other direction.

Suppose that \eqref{equation: DE 3} is apparent throughout $\H^\ast$.
Let $y_{\pm}(z)=q^{\pm n_\infty/2}\left(1+O(q)\right)$ be two solutions of
\eqref{equation: DE 3}
and set $h(z)=y_+(z)/y_-(z)$. Since \eqref{equation: DE 3} is apparent
on $\H$, $h(z)$ is a meromorphic single-valued function on $\H$ and
its Schwarz derivative is $-2Q_3(z)$. Recall Bol's theorem that there
is a homomorphism $\rho:\SL(2,\Z)\rightarrow\rm{PSL}(2,\C)$ such that
$$\begin{pmatrix}
\left(y_1\big|_{-1}\gamma\right)(z)\\
\left(y_2\big|_{-1}\gamma\right)(z)
\end{pmatrix}=\pm\rho(\gamma)\begin{pmatrix}
y_1(z)\\
y_2(z)
\end{pmatrix},\quad \gamma\in\SL(2,\C).
$$ 
 Clearly, $\rho(T)=\pm I$ because $\infty$ is apparent. Note that
 $\ker\rho$ is a normal subgroup of $\SL(2,\Z)$ and contains $\gamma
 T\gamma^{-1}$ for any $\gamma\in\SL(2,\Z)$. In particular, $\ker\rho$
 contains both $T=\SM1101$ and $STS^{-1}=\SM10{-1}1$, where
 $S=\SM0{-1}10$. Since $\SM1101$ and $\SM10{-1}1$ generate
 $\SL(2,\Z)$, we conclude that $\ker\rho=\SL(2,\Z)$.
% A well-known theorem [??] ensures such a subgroup must be identical
% with the full group $\SL(2,\Z)$.
In other words, $\rho(\gamma)=\pm I$ and $h(z)$ is a modular function
on $\SL(2,\Z)$. Thus we have a solution
$u:=\log\frac{8\abs{h'(z)}^2}{\left(1+\abs{h(z)}^2\right)^2}$ which
realizes $Q_3$. This proves the claim.
 
 Now, we apply the Eremenko theorem with the angles given by
 $\theta_1=\kappa_i$, $\theta_2=2\kappa_\rho/3$ and
 $\theta_3=n_\infty$. Our necessary and sufficient condition in
 Theorem \ref{theorem: necessary and sufficient conditions for appartness at infinity} is identically the same as the condition of
 Eremenko's theorem for the existence of $u$ with three integral
 angles. This proves Theorem \ref{theorem: necessary and sufficient conditions for appartness at infinity}.
 \end{proof}

\begin{Theorem}\label{theorem: condition for Q be realized}
Suppose $\kappa_i\in\N$ and $\kappa_\rho,\kappa_\infty\in\frac{1}{2}\N$ such that $2\kappa_\rho/3\in\N$.
If $Q_3(z;r,s,t)$ is apparent at $\rho$ and $i$, then $Q$ can be realized.
\end{Theorem}

\begin{proof}
By the assumption, we have that $\theta_i$, $1\leq i\leq 3$, are all integers. Now, given $\kappa_i$ and $\kappa_\rho$, $s$ and $t$ are determined by the same formula in our paper. Further, there are polynomials $P_1$ and $P_2$:
\begin{enumerate}
\item[$\bullet$]
$Q_3(z;r,s,t)$ is apparent at $i$ if and only if $P_1(r)=0$, and $\deg P_1(r)=\kappa_i$.
\item[$\bullet$]
$Q_3(z;r,s,t)$ is apparent at $\rho$ if and only if $P_2(r)=0$, and $\deg P_2=2\kappa_\rho/3$.

\end{enumerate}
%We claim that if $\deg P_i\leq\deg P_j$, then $P_i$ is a factor of $P_j$.
Therefore, $Q_3(z;r,s,t)$ is apparent at $i$ and $\rho$ if and only if
$$
r\in\braces{r:P_1(r)=P_2(r)=0}.
$$
Now, we claim that under the assumption $\theta_1\in\N$,
$Q_3(z;r,s,t)$ is apparent if and only if the local 
exponents at $\infty$ are $\pm \kappa_\infty/2$, $\kappa_\infty\in\N$
and the curvature equation has a solution.

By Eremenko's Theorem (Section 2.4), (recall $\theta_1=\kappa_i$,
$\theta_2=2\kappa_\rho/3$, $\theta_3=2\kappa_\infty$) the curvature
equation has a solution if and only if $\theta_1+\theta_2+\theta_3$ is
odd and $\theta_i<\theta_j+\theta_k$, $i\neq j\neq k$. This condition
is equivalent to
\begin{enumerate}
\item[(a)] $$\theta_2-\theta_1<\theta_3<\theta_2+\theta_1,\quad \text{and}
$$
\item[(b)] $$\theta_1-\theta_2<\theta_3<\theta_1+\theta_2.$$
\end{enumerate}
Since $\theta_1+\theta_2+\theta_3$ is odd, we have $\theta_2$
solutions of the curvature equation if $\theta_1>\theta_2$, $\theta_1$
solutions if $\theta_2>\theta_1$.

Now, $\deg P_1=\kappa_i=\theta_1$ and $\deg P_2=2\kappa_\rho/3=\theta_2$.
Then
\begin{align*}
\min\braces{\theta_1,\theta_2}&\geq\abs{\braces{r:P_1(r)=P_2(r)=0}}\\
&=\ge\#\ \text{of\ curvature\ equations}
\geq\min\braces{\theta_1,\theta_2}.
\end{align*}
Thus
$$
\abs{\braces{r:P_1(r)=P_2(r)=0}}=\#\ \text{of\ curvature\ equations}.
$$
This proves the theorem.
%implies if $Q_3(z;r,s,t)$ is apparent and $-Q(\infty)/(4\pi^2)=(n\kappa_\infty/2)^2$ for $\kappa_\infty\in\N$, then $Q_3$ is apparent at $\infty$, a contradiction to our assumption that $Q$ is not apparent at $\infty$. Thus we have $\kappa_i\in\frac{1}{2}+\N$, i.e., the local exponents at $i$ are integers. Therefore, $y(z)\sqrt{E_4(z)}$ is well-defined on $\H$.
\end{proof}

\begin{Remark}
In fact, the proof shows that if $\deg P_i\leq\deg P_j$, then $P_i$ is
a factor of $P_j$.
\end{Remark}

\section{proof of Theorem \ref{theorem: 1.1} and Theorem
  \ref{theorem: 1.2}}
\label{section: proof of theorem 1.1}

\begin{proof}[Proof of Theorem \ref{theorem: 1.1}]
Let $\rho$ be the Bol representation associated to \eqref{(1.1)}, and
set $T=\SM 1101$, $S=\SM 0{-1}10$, and $R=TS=\SM 1{-1}10$. They satisfy
\begin{equation}\label{equation: S^2 and R^3}
 S^2=-I,\quad \text{and}\quad R^3=-I.
\end{equation}
Assume that ($\mb H_1$) and ($\mb H_2$) hold. It follows from either
\cite[Theorem 2.5]{Eremenko-Tarasov}, quoted as Theorem \ref{theorem:
  ET} in the appendix, or Theorem \ref{thm1} (with
$\theta_1=1/2$, $\theta_2=1/3$, and $\theta_3=2r_\infty$ or
$\theta_3=1-2r_\infty$, depending on whether $2r_\infty\le1/2$ or
$2r_\infty>1/2$) in the appendix that if $1/12<r_\infty<5/12$, then an
invariant metric realizing $Q(z)$ exists, and if $0<r_\infty<1/12$ or
$5/12<r_\infty\le1/2$, then there does not exist an invariant metric
realizing $Q(z)$. So here we are concerned with the case $r_\infty=1/12$ or
$r_\infty=5/12$.

Assume that $r_\infty=1/12$. Then there exists a basis
$\{y_1(z),y_2(z)\}$ for the
solution space of \eqref{(1.1)} such that
\begin{equation} \label{equation: rho(T) 2}
\rho(T)=\pm\M{\epsilon}00{\overline\epsilon}, \qquad \epsilon=e^{2\pi i/12}.
\end{equation}
Since $S^2=-I$, we have $\rho(S)^2=\pm I$. The matrix $\rho(S)$ cannot
be equal to $\pm I$ as the relation $R=TS$ will imply that the
eigenvalues of $\rho(R)$ are $\pm e^{2\pi i/12}$ or $\pm e^{-2\pi
  i/12}$, which is absurd. It follows that $\tr\rho(S)=0$ and we have
\begin{equation} \label{equation: rho(S) 2}
\rho(S)=\pm\M abc{-a}, \qquad
\rho(R)=\pm\rho(T)\rho(S)
=\pm\M{\epsilon a}{\epsilon b}{\bar{\epsilon}c}{-a\bar{\epsilon}}
\end{equation}
for some $a,b,c\in\C$. Since $\rho(R)^3=\pm I$, $\det\rho(R)=1$, and
$\rho(R)\neq\pm I$ by a similar reason as above, the characteristic
polynomial of $\rho(R)$ has to be $x^2-x+1$ or $x^2+x+1$. In
particular, we have $\tr\rho(R)=\pm 1$, i.e.,
$a(\epsilon-\overline\epsilon)=\pm 1$ and hence $a=\pm i$ and $bc=0$.
Under the assumption that there is an invariant metric realizing
$Q(z)$, the matrices $\rho(S)$, $\rho(T)$, and $\rho(R)$ must be
unitary, after a simultaneous conjugation. (See the discussion in
Section 2.2.) If one of $b$ and $c$ is not $0$, this cannot happen.
Therefore, we have $b=c=0$. This implies that the function $y_1(z)^2$,
which is meromorphic throughout $\H$ since the local exponents at
every singularity are in $\frac12\Z$, satisfies
$$
y_1(Tz)^2=e^{2\pi i/6}y_1(z)^2, \qquad
y_1(Sz)^2=-z^{-2}y_1(z)^2.
$$
It follows that $y_1(z)^2$ is a meromorphic modular form of weight
$-2$ with character $\chi$ on $\SL(2,\Z)$. Likewise, we can show that
$y_2(z)^2$ is a meromorphic modular form of weight $-2$ with character
$\overline\chi$. This proves that if there is an invariant metric
realizing $Q(z)$, then there are solutions $y_1(z)$ and $y_2(z)$ with
the stated properties. The proof of the case
$r_\infty=5/12$ is similar and is omitted.

The proof of the converse statement is easy. If there exist solutions
$y_1(z)$ and $y_2(z)$ of \eqref{(1.1)} such that $y_1(z)^2$ and
$y_2(z)^2$ are meromorphic modular forms of weight $-2$ with character
$\chi$ and $\overline\chi$, respectively, on $\SL(2,\Z)$, then
$y_1(Tz)^2=e^{2\pi i/6}y_1(z)^2$ and $y_2(Tz)^2=e^{-2\pi
  i/6}y_2(z)^2$, which implies that $y_1(z)^2$ and $y_2(z)^2$ are of
the form
$y_1(z)^2=q^{1/6}\sum_{j\ge n_0}c_jq^j$ and
$y_2(z)^2=q^{-1/6}\sum_{j\ge n_0}d_jq^j$.
It follows that $r_\infty=1/12$ or $r_\infty=5/12$. It is clear that
with respect to the basis $\{y_1(z),y_2(z)\}$, the Bol representation is given by
$$
\rho(T)=\pm\M{e^{2\pi i/12}}00{e^{-2\pi i/12}}, \qquad
\rho(S)=\pm\M{\pm i}00{-i},
$$
and hence is unitary. It follows that there is an invariant metric of
curvature $1/2$ realizing $Q(z)$. This proves the theorem.
\end{proof}

We now give two examples with $r_\infty=1/12$, one of which can be
realized by some invariant metric of curvature $1/2$, while the other
of which can not. Note that Theorem 1 of \cite{Eremenko} implies that
when \eqref{(1.1)} does not have $\SL(2,\Z)$-inequivalent
singularities outside $\{i,\rho\}$, $1/12<r_\infty<5/12$ is the
necessary and sufficient condition for the existence of an invariant
metric of curvature $1/2$ realizing $Q$. The examples we provide below
show that when \eqref{(1.1)} has $\SL(2,\Z)$-inequivalent
singularities other than $i$ and $\rho$, this condition is no longer
a necessary condition.

\begin{Example} Let
  $\eta(z)=q^{1/24}\prod_{n=1}^\infty(1-q^n)=\Delta(z)^{1/24}$,
  \begin{equation} \label{equation: xy}
  x(z)=\frac{E_4(z)}{\eta(z)^8}=q^{-1/3}+\cdots, \qquad
  y(z)=\frac{E_6(z)}{\eta(z)^{12}}=q^{-1/2}+\cdots,
  \end{equation}
  and $h(z)=x(z)/y(z)=q^{1/6}+\cdots$. They are modular functions on
  the unique normal subgroup $\Gamma$ of $\SL(2,\Z)$ of index $6$ such
  that $\SL(2,\Z)/\Gamma$ is cyclic. (Another way to describe $\Gamma$
  is that $\Gamma=\ker\chi$, where $\chi$ is the character of
  $\SL(2,\Z)$ such that $\chi(S)=-1$ and $\chi(R)=e^{2\pi i/3}$.)
  Using Ramanujan's identities
  \begin{equation*}
    \begin{split}
      D_qE_2(z)&=\frac{E_2(z)^2-E_4(z)}{12}, \\
      D_qE_4(z)&=\frac{E_2(z)E_4(z)-E_6(z)}3, \\
      D_qE_6(z)&=\frac{E_2(z)E_6(z)-E_4(z)^2}2,
    \end{split}
  \end{equation*}
  where $D_q=qd/dq$ (see \cite[Proposition 15]{Zagier123}) and the
  relation $\Delta(z)=(E_4(z)^3-E_6(z)^2)/1728$, we can compute that
  $$
  \{h(z),z\}=(2\pi i)^2Q_0(z)
  $$
  where
  $$
  Q_0(z)=E_4(z)\left(
    -\frac1{72}-\frac{9(E_4(z)^3-E_6(z)^2)^2}{(3E_4(z)^3-2E_6(z)^2)^2}
    +\frac52\frac{E_4(z)^3-E_6(z)^2}{3E_4(z)^3-2E_6(z)^2}\right).
  $$
  Thus,
  $$
  y_+(z)=\frac{h(z)}{\sqrt{D_qh(z)}}=q^{1/12}+\cdots, \quad
  y_-(z)=\frac1{\sqrt{D_qh(z)}}=q^{-1/12}+\cdots
  $$
  are solutions of the differential equation
  $y''(z)=Q(z)y(z)$, where $Q(z)=-(2\pi i)^2Q_0(z)/2$. The meromorphic
  modular form $Q(z)$ has only one $\SL(2,\Z)$-inequivalent
  singularity at the point $z_1$ such that $3E_4(z_1)^3-2E_6(z_1)^2=0$
  and is holomorphic at the elliptic points $i$ and $\rho$. In the
  notation of Theorem \ref{theorem: 1.1}, we have $r_\infty=1/12$.
  This provides an example of an invariant metric of curvature $1/2$
  realizing a meromorphic modular form of weight $4$ with a threshold
  $r_\infty$. Note that with respect to the basis $\{y_+,y_-\}$, the
  Bol representation is given by
  $$
  \rho(T)=\pm\M{e^{2\pi i/12}}00{e^{-2\pi i/12}}, \qquad
  \rho(S)=\pm\M i00{-i},
  $$
  both of which are unitary. (The information about $\rho(S)$ follows
  from the transformation formula $\eta(-1/z)=\sqrt{z/i}\eta(z)$ and
  the fact that $D_qh(z)=C\eta(z)^4(3E_4(z)^3-2E_6(z)^2)/E_6(z)^2$ for
  some constant $C$.)
\end{Example}

\begin{Example}
  Let $x(z)$ and $y(z)$ be defined by \eqref{equation:
    xy}, and $\Gamma$ be the unique normal subgroup of $\SL(2,\Z)$ of
  index $6$ such that $\SL(2,\Z)/\Gamma$ is cyclic. The modular curve
  $X(\Gamma):=\Gamma\backslash\H^\ast$
  has one cusp of width $6$, no elliptic points, and is of genus $1$.
  Since the modular functions $x(z)$ and $y(z)$ on $\Gamma$ have only
  a pole of order $2$ and $3$, respectively, at the cusp $\infty$ and
  are holomorphic elsewhere, they generate the function field of
  $X(\Gamma)$. Then from the relation
  $E_4(z)^3-E_6(z)^2=1728\eta(z)^{24}$, we see that $x(z)$ and $y(z)$
  satisfies
  $$
  y^2=x^3-1728,
  $$
  which we may take as the defining equation for $X(\Gamma)$.
%  Let $\omega$ be a meromorphic $1$-form of the second kind on the
%  elliptic curve $y^2=x^3-1728$ (i.e., all residues are $0$). Since such a
%  meromorphic $1$-form can be
%  expressed as $\omega=f(z)\,dz$ for some meromorphic modular form
%  $f(z)$ of weight $2$ on $\Gamma$, where $\Gamma$ is the normal
%  subgroup of index $6$ of $\SL(2,\Z)$ mentioned in the previous
%  example.
  Let $f(z)$ be a meromorphic modular form of weight $2$ on $\Gamma$
  such that all residues on $\H$ are $0$. Equivalently, let $\omega=f(z)\,dz$
  be a meromorphic differential $1$-form of the second kind on
  $X(\Gamma)$. Consider
  $$
  y_1(z)=\frac1{\sqrt{f(z)}}\int_{z_0}^zf(u)\,du, \qquad
  y_2(z)=\frac1{\sqrt{f(z)}},
  $$
  where $z_0$ is a fixed point in $\C$ that is not a pole of $f(z)$.
  Under the assumption that all residues of $f(z)$ are $0$, the
  integral in the definition of $y_1(z)$ does not depend on
  the choice of path of integration from $z_0$ to $z$.
  A straightforward computation shows that the Wronskian of $y_1$ and
  $y_2$ is a constant and hence $y_1(z)$ and $y_2(z)$ are solutions of
  the differential equation $y''(z)=Q(z)y(z)$, where
  $$
  Q(z)=\frac{3f'(z)^2-2f(z)f''(z)}{4f(z)^2}
  $$
  can be shown to be a meromorphic modular form of weight $4$ on
  $\Gamma$. (The numerator of $Q(z)$ is a constant mulitple of the
  Rankin-Cohen bracket $[f,f]_2$ and hence a mermomorphic modular form
  of weight $8$. See \cite{Rankin-Cohen}.) By construction, this
  differential equation is apparent throughout $\H$. Furthermore, if
  $f(z)$ is chosen in a way such that
  $f(\gamma z)=\chi(\gamma)(cz+d)^2f(z)$ holds for all $\gamma=\SM
  abcd\in\SL(2,\Z)$ for some character $\chi$ of $\SL(2,\Z)$ with
  $\Gamma\subset\ker\chi$, then $Q(z)$ is modular on $\SL(2,\Z)$. We
  now utilize this construction of modular differential equations to find
  $Q(z)$ that cannot be realized, i.e., the monodromy group is not
  unitary.

  We let $\omega_1=dx/y$ and $\omega_2=d(x/y^3)$. Note that $\omega_1$
  is a holomorphic $1$-form on the curve $y^2=x^3-1728$, while $\omega_2$
  is an exact $1$-form and hence a meromorphic $1$-form of the
  second kind. Using Ramanujan's identities, we check that
  $\omega_1=f_1(z)\,dz$ and $\omega_2=f_2(z)\,dz$ with
  $$
  f_1(z)=-\frac{2\pi i}3\eta(z)^4, \quad
  f_2(z)=2\pi i\frac{\eta(z)^4}{E_6(z)^4}\left(\frac76E_4(z)^3\Delta(z)
    +576\Delta(z)^2\right).
  $$
  Now we choose, say,
  $$
  \omega=-\frac{3}{2\pi i}\left(\omega_1+\omega_2\right)
  $$
  and let $f(z)=q^{1/6}+\cdots$ be the meromorphic modular form of
  weight $2$ such that $\omega=f(z)\,dz$. Let $y''(z)=Q(z)y(z)$ be the
  differential equation obtained from $f(z)$ using the construction
  described above. Note that $f(z+1)=e^{2\pi i/6}f(z)$ and using
  $\eta(-1/z)=\sqrt{z/i}\eta(z)$, we have $f(-1/z)=-z^2f(z)$. Thus,
  $f(\gamma z)=\chi(\gamma)(cz+d)^2f(z)$ for all $\gamma=\SM
  abcd\in\SL(2,\Z)$, where $\chi$ is the character of $\SL(2,\Z)$ such
  that $\chi(T)=e^{2\pi i/6}$ and $\chi(S)=-1$. According the
  discussion above, the function $Q(z)$ is a meromorphic modular form
  of weight $4$ with trivial character on $\SL(2,\Z)$. Note that
  $f(z)$ has zeros at points where
  $6E_6(z)^4-7E_4(z)^3\Delta(z)-3456\Delta(z)^2=0$.
  Now let us compute its Bol representation.

  We choose $z_0=i\infty$ and find that
  $$
  y_2(z)=q^{-1/12}\left(1+\sum_{j=1}^\infty c_jq^j\right), \qquad
  y_1(z)=q^{1/12}\sum_{j=0}^\infty d_jq^j
  $$
  for some $c_j$ and $d_j$ with $d_0\neq 0$. Therefore, the local
  exponents at $\infty$ are $\pm1/12$ and
  $$
  \rho(T)=\pm\M{e^{2\pi i/12}}00{e^{-2\pi i/12}}.
  $$
  Also, since $f(-1/z)=-z^2f(z)$, we have
  \begin{equation*}
    \begin{split}
      \int_{i\infty}^{-1/z}f(u)\,du
      &=\int_0^zf(-1/u)\,\frac{du}{u^2}
      =-\int_0^zf(u)\,du \\
      &=-\int_0^{i\infty}f(u)\,du-\int_{i\infty}^zf(u)\,du.
    \end{split}
  \end{equation*}
  Thus,
  $$
  \rho(S)=\pm\M iC0{-i}, \qquad
  C=i\int_0^{i\infty}f(u)\,du.
  $$
  Now recall that $\omega=f(z)\,dz$ is equal to
  $-3(\omega_1+\omega_2)/(2\pi i)$. Since $\omega_2=d(x/y^3)$ is an
  exact $1$-form on $X(\Gamma)$ and the modular curve $X(\Gamma)$ has
  only one cusp, which in particular says that $\infty$ and $0$ are
  mapped to the same point on $X(\Gamma)$ under the natural map
  $\H^\ast\to X(\Gamma)$, the integral $\int_0^{i\infty}f_2(u)\,du$ is
  equal to $0$. Therefore, we have
  $$
  C=i\int_0^{i\infty}\eta(u)^4\,du.
  $$
  This constant $C$ can be expressed in terms of the central value of
  the $L$-function of the elliptic curve $E:y^2=x^3-1728$, which is
  known to be nonzero. From this, it is straightforward to check
  that there is no simultaneous conjugation such that $\rho(T)$ and
  $\rho(S)$ both become unitary.
\end{Example}

\begin{proof}[Proof of Theorem \ref{theorem: 1.2}]
  We use the notations in the proof of Theorem \ref{theorem:
    1.1}. Since $\kappa_\infty=n/4$ for some odd integer $n$, with
  respect to the basis $\{y_+(z),y_-(z)\}$, we have $\rho(T)=\pm\SM
  i00{-i}$. If $\rho(S)=\pm I$, then $\rho(R)=\pm\SM i00{-i}$, which
  is a contradiction to $\rho(R)^3=\pm I$. Hence, $\rho(S)\neq\pm I$,
  and we have $\tr\rho(S)=0$. Then, by a choosing a suitable scalar $r$, the
  matrix of $\rho(S)$ with respect to 
  $\{ry_+(z),y_-(z)\}$ will be of the form
  \begin{equation*} %\label{equation: rho(S)= pm i M}
  \rho(S)=\pm\M abb{-a}
  \end{equation*}
  for some $a,b\in\C$ with $a^2+b^2=-1$, while $\rho(T)$ is still
  $\pm\SM i00{-1}$. Set $F(z)=r^2y_+(z)^2+y_-(z)^2$. We then compute that
  $F(Tz)=-F(z)$ and
  \begin{equation*}
    \begin{split}
      \left(F|_{-2}S\right)(z)&=(ary_+(z)+by_-(z))^2+(bry_+(z)-ay_-(z))^2
      \\
      &=-r^2y_1(z)^2-y_2(z)^2=-F(z).
    \end{split}
  \end{equation*}
  This proves the theorem. 
\end{proof}
 
\section{Existence of the curvature equation}
In this section, we will prove Theorem \ref{theorem: existence of modular form Q with described local exponents and apparentness} equipped with the data \eqref{equation: datas}.
The main purpose of this section is to prove the existence and the number of such $Q$ equipped with data \eqref{equation: datas}. The discussion will be divided into several cases depending on $\kappa_\rho$ and $\kappa_i$.

\begin{Lemma}\label{lemma: F has at most simple zero}
Suppose $F(z)$ is a modular form of weight $4$ with respect to $\SL(2,\Z)$, and is holomorphic except at $\rho$ and $i$. If the pole order of $F(z)$ at $\rho$ or $i$ $\leq 1$, 
then $F(z)$ is holomorphic.
\end{Lemma}
\begin{proof}
Let $n_1$ and $n_2$ be the orders of poles at $i$ and $\rho$ respectively. The counting zero formula of meromorphic modular form (see \cite{Serre}) says
$$
m-\frac{n_1}{2}-\frac{n_2}{3}=\frac{4}{12},\qquad m\ \text{is a non-negative integer}.
$$
By the assumption, $n_i\leq 1$. From the identity, it is easy to see $n_1\leq 0$ and $n_2\leq 0$.
\end{proof}

%\begin{Lemma}\label{lem: Fj simple zero}
%$F_j:=E_6(z)^2-t_jE_4(z)^3$ has zeros
%only at $z_j$, and $z_j$ is a simple zero.
%\end{Lemma}
%\begin{proof}
%Let $'$ denote $\frac{1}{2\pi i}\frac{d}{dz}$. The Ramanujan differential identities yield
%$$
%(E_6(z)^2-t_jE_4(z))'=(E_6(z)^2-t_jE_4(z)^3(z))E_2(z)+(t_j-1)E_4(z)^2E_6(z).
%$$
%The lemma follows immediately from this identities and the fact that $E_4(z)$ and $E_6(z)$ do not have common zero.
%\end{proof}

Let $t_j=E_6(z_j)^2/E_4(z_j)^3$ and define $F_j(z)=E_6(z)^2-t_jE_4(z)^3$. By the theorem of counting zeros of modular forms \cite[p. 85, Theorem 3]{Serre}, $F_j(z)$ has a (simple) zero at $z_j\in\H$.

\begin{Lemma} \label{lemma: polynomials} Suppose that $Q$ satisfies
  the conditions (i) and (ii) 
  in Definition \ref{definition: Q is equipped}. Then
\begin{equation}\label{equation: shape of Q}
\begin{split}
&Q=\pi^2\left( Q_3(z;r,s,t)
+\sum^m_{j=1}\frac{r_1^{(j)}E_4(z)^4F_j(z)+r^{(j)}_2E_4(z)^7}{F_j(z)^2}\right),
\end{split}
\end{equation}
where $r$, $r^{(j)}_1$ are free parameters and $s$, $t$, $r^{(j)}_2$ are uniquely determined by
\begin{equation}
\begin{split}
&s=s_{\kappa_\rho}:=(1-4\kappa^2_\rho)/9,\quad t=t_{\kappa_i}=(1-4\kappa^2_i)/4,\quad \text{and}\\
& r^{(j)}_2=r^{(j)}_{2,\kappa_j}=t_j(t_j-1)^2(1-4\kappa^2_j)/4.
\end{split}
\end{equation}
\end{Lemma}
\begin{proof}
Let $\hat{Q}$ denote the RHS of \eqref{equation: shape of Q}. Then it is a straightforward computation to show that (ii) in Definition \ref{definition: Q is equipped} holds at $p_j$ if and only if $s=s_{\kappa_\rho}$ if $p_j=\rho$, $t=t_{\kappa_i}$ if $p_j=i$, and $r^{(j)}_2=r^{(j)}_{2,\kappa_j}$ if $p_j=z_j$. By the choice of $s$, $t$ and $r^{(j)}_2$,
$Q-\hat{Q}$ might contain simple poles only.
Further, we can choose $r^{(j)}_1$ to make $Q-\hat{Q}$ holomorphic at $z_j$. By Lemma \ref{lemma: F has at most simple zero}, $Q-\hat{Q}$ is automatically smooth at $\rho$ and $i$. Therefore, $Q-\hat{Q}$ is a holomorphic modular form of weight $4$, and the lemma follows immediately because $E_4(z)$, up to a constant, is the only holomorphic modular form of weight $4$.
\end{proof}

%\begin{Theorem}\label{theorem: apparent case I}
%Given \eqref{equation: datas 2} such that $\kappa_i\not\in\N$ and $2\kappa_\rho/3\not\in\N$. Then there are exactly $N$ tuples $(r_\ell,r^{(j)}_{1,\ell})$ with multiplicity, $\ell=1,2,\cdots,N$,
 %such that
%$Q\left(z;r_\ell,s_{\kappa_\rho},t_{\kappa_i},r_{1,\ell}^{(j)},r^{(j)}_{2,\kappa_j}\right)$
%is equipped with \eqref{equation: datas 2}.
%\end{Theorem}

Now we are in the position to prove Theorem \ref{theorem: existence of modular form Q with described local exponents and apparentness}.
 \begin{proof}[Proof of Theorem \ref{theorem: existence of modular form Q with described local exponents and apparentness}]
We first calculate the parameters $r,r^{(j)}_1$, $1\leq j\leq m$, such that $Q$ is apparent at $z_j$. For simplicity, we assume $j=1$.
From \eqref{equation: shape of Q}, we do the Taylor expansion at $z=z_1$.
\begin{align*}
&Q(z)=a_{-2}(z-z_1)^{-2}+(r_1b_{-1}+a_{-1})(z-z_1)^{-1}\\
&+\sum^\infty_{j=0}\left(a_j+r_1b_j+c_j\left(r,r^{(2)}_1,\ldots,r^{(m)}_1\right)\right)\left(z-z_1\right)^j:=\sum^\infty_{j=-2}A_j\left(z-z_1\right)^j,
\end{align*}
where $a_j,b_j$ are independent of $r$, $r^{(j)}_1$ and $c_j(r,r^{(2)}_1,\ldots,r^{(m)}_1)$ is linear in all variables, and also
$$
y(z)=(z-z_1)^{1/2-\kappa_1}\left(1+\sum^\infty_{j=1}d_j(z-z_1)^j\right).
$$
Then we derive the recursive formula by comparing both sides of
\eqref{(1.1)} with $Q$ in \eqref{equation: shape of Q},
\begin{equation}
j(j-2\kappa_1)d_j=\sum_{k+\ell=j-2,\ k<j}d_kA_\ell,\quad A_{-1}=a_{-1}+r_1b_{-1},
\end{equation}
where $d_0=1$ and
$$
d_1=\frac{1}{1-2\kappa_1}d_0A_{-1}=\frac{b_{-1}}{1-2\kappa_1}r_1+\text{terms\ of\  lower\ orders}.
$$
By induction,
\begin{align}\label{equation: recursive formula}
&j(j-2\kappa_1)d_j=d_{j-1}A_{-1}+d_{j-2}A_0+d_{j-3}A_1+\cdots+d_0A_{j-2}\\
&\nonumber=\frac{b^{j-1}_{-1}}{(1-2\kappa_1)\cdots\left((j-1)-2\kappa_1\right)}r^{j-1}_1+\ \text{terms\ of\ lower\ orders}.
\end{align}
At $j=2k_1$, the RHS of \eqref{equation: recursive formula} is
$$
P_1\left(r,r^{(1)}_1,\ldots,r^{(m)}_1\right):=d_{2\kappa_1-1}A_{-1}+d_{2\kappa_1-2}A_0+\cdots+d_0A_{\kappa_1-2}.
$$
Clearly, $\deg P_1=2\kappa_1$ and 
\begin{equation}\label{equation: P_1}
P_1=B_0r^{2\kappa_1}_1+\text{terms\ of\ lower\ orders},\quad B_0\neq 0.
\end{equation}
We summarized what are known:
\begin{enumerate}
\item[$\bullet$]
$\kappa_i\not\in\N$, then $Q$ is apparent at $i$ for any tuple $\left(r,r^{(j)}_1\right)$.
\item[$\bullet$]
$2\kappa_p/3\not\in\N$, then $Q$ is apparent at $\rho$ for any tuple $\left(r,r^{(j)}_1\right)$,
\item[$\bullet$]
$1/2\pm\kappa_j$, there is a polynomial $P_j\left(r,r^{(1)}_1,\ldots,r^{(m)}_1\right)$ of degree $2\kappa_j$ such that
$Q$ is apparent at $z_j$ if and only if $P_j\left(r,r^{(1)}_1,\ldots,r^{(m)}_1\right)=0$. 
\end{enumerate}

Since $\kappa_\infty$ is given, we have $\kappa_\infty=\sqrt{-Q(\infty)}/2$, and then
\begin{equation}\label{equation: relation of r and r_j}
r+\sum^m_{j=1}\left(1-t_j\right)r^{(1)}_j+e=0,
\end{equation}
where $e$ is given. By Bezout's theorem, we have $N=\prod^m_{j=1}(2\kappa_j)$ common roots with multiplicity of
\eqref{equation: P_1} and \eqref{equation: relation of r and r_j}
because by \eqref{equation: P_1} it is easy to see that there are
no solutions at $\infty$.
This proves the theorem.

\end{proof}

\appendix
\section{Curvature equations on $S^2$ with multiple singularities}
Let $\H^\ast=\H\cup\Q\cup\{\infty\}$. Since
$\SL(2,\Z)\backslash\H^\ast\simeq\C\cup\{\infty\}$, the equation
\eqref{2.4} in the case $\Gamma=\SL(2,\Z)$ can be transformed into the
mean field equations on $\C$:
\begin{equation}\label{eq1}
\begin{cases}
\Delta u+e^u=4\pi\left(\alpha_1\delta_{0}+\alpha_2\delta_{1}+\sum^m_{j=1}n_j\delta_{p_j}\right)\quad\text{on}\ \C,\\
u(z)=-(4+2\alpha_3)\log\abs z+O(1)\quad\text{as}\ \abs z\rightarrow\infty,
\end{cases}
\end{equation}
where we assume that the isomorphism maps the points $i=\sqrt{-1}$,
$\rho=(1+\sqrt{-3})/2$, and $\infty$ of $\SL(2,\Z)\backslash\H^\ast$ to $0$, $1$,
and $\infty$, respectively, $\delta_p$ is the Dirac measure at
$p\in\C$, $\alpha_k>-1$ for
$k=1,2,3$ and $n_j\in\N$. For any solution $u$ of \eqref{eq1}, the
conformal metric $e^u|dz|^2$ has the angles $\lambda_1$, $\lambda_2$,
and $\sigma_j$ at $0$, $1$, and $p_j$, respectively, where
\begin{equation} \label{A.2}
  \lambda_1=\alpha_1+1, \qquad \lambda_2=\alpha_2+1, \qquad
  \sigma_j=n_j+1.
\end{equation}

Throughout the appendix, we assume that $\alpha_k$ are not integers for
$k=1,2,3$ and all $p_j$ are distinct. To find a solution for
\eqref{eq1}, we first associate to \eqref{eq1} a second-order ODE
\begin{equation} \label{A.3}
y''(z)+Q(z)y(z)=0, \qquad z\in\C,
\end{equation}
where
\begin{align}\label{eq11}
  Q(z)=&\left(\frac{\tfrac{\alpha_1}{2}(\tfrac{\alpha_1}{2}+1)}{z^2}
         +\frac{r_1}{z}\right)+
         \left(\frac{\tfrac{\alpha_2}{2}(\tfrac{\alpha_2}{2}+1)}{(z-1)^2}
         +\frac{r_2}{z-1}\right)\\
       &\qquad+\sum_{j=1}^m\frac{\tfrac{n_j}{2}(\tfrac{n_j}{2}+1)}{(z-p_j)^2}
         +\frac{s_j}{z-p_j}\nonumber
\end{align}
for some free parameters $r_0, r_1, s_j$. It is known that \eqref{eq1}
has a solution if and only if the monodromy group of \eqref{A.3} is
projectively unitary.

Note that the local exponents of \eqref{A.3} at $0$
and $1$ are $\{-\alpha_1/2,1+\alpha_1/2\}$ and
$\{-\alpha_2/2,1+\alpha_2/2\}$, respectively. Since
$\alpha_1,\alpha_2\notin\Z$, the differences of the local exponents
are not integers. At each $p_j$, there is a polynomial
$P_j(r_1,r_2,s_j)$ such that \eqref{A.3} is apparent if and only if
$P_j(r_1,r_2,s_j)=0$. The derivation of the polynomials $P_j$ is the
same as Lemma \ref{lemma: polynomials}. Moreover, the asymptotic
behavior of $u$ at $\infty$ yields that \eqref{A.3} is Fuchsian at
$\infty$ with local exponents $-\alpha_3/2$ and $1+\alpha_3/2$. Thus,
we have
$$
r_1+r_2+\sum_js_j=0
$$
and
\begin{equation*}%\label{eq13}
  \begin{split}
\frac{\alpha_\infty}{2}(\frac{\alpha_\infty}{2}+1)=&\lim_{z\to\infty}z^2Q(z)\\
=&r_1+\sum_{j=1}^m s_jp_j+\sum_{k\in\{0,1\}}\frac{\alpha_k}{2}(\frac{\alpha_k}{2}+1)
+\sum_{j=1}^m\frac{n_j}{2}(\frac{n_j}{2}+1).
  \end{split}
\end{equation*}
Therefore, for given local exponent data for \eqref{eq1}, the B\'ezout theorem
implies that there are at most $\prod_{j=1}^m(n_j+1)$ distinct $Q$ such that
\eqref{A.3} realizes the mean field equation \eqref{eq1} for given
data. Theorem 2.5 of \cite{Eremenko-Tarasov} is to give a necessary
and sufficient condition to ensure that the projective monodromy group
of \eqref{A.3} is unitary, i.e., that \eqref{eq1} has a solution.

\begin{Theorem}[{\cite[Theorem 2.5]{Eremenko-Tarasov}}]
  \label{theorem: ET}
  Suppose that $\alpha_1,\alpha_2,\alpha_3$ are not integers and all
  combinations
  \begin{equation} \label{A.4}
  \alpha_1\pm\alpha_2\pm\alpha_3\text{ are not integers}
  \end{equation}
  for any choice of signs. Then \eqref{eq1} has a solution if and only if
  $$
  \cos^2\pi\alpha_1+\cos^2\pi\alpha_2+\cos^2\pi\alpha_3
  +2(-1)^{\sigma+1}\cos\pi\alpha_1\cos\pi\alpha_2\cos\pi\alpha_3<1,
  $$
  where $\sigma=\sum_{j=1}^m n_j$. Moreover, the number of distinct
  solutions of \eqref{eq1} is less than or equal to $\prod_{j=1}^m(n_j+1)$.
\end{Theorem}

We remark that the notations $\alpha_j$ here differ from those
used in \cite{Eremenko-Tarasov} by $1$.

Note that when \eqref{eq1} arises from the differential equation
\eqref{(1.1)} considered in Theorem \ref{theorem: 1.1}, we have
$$
\alpha_1=\kappa_i-1, \quad \alpha_2=2\kappa_\rho/3-1, \quad
\alpha_3=2\kappa_\infty, \quad
n_j=2\kappa_{p_j}-1,
$$
where $\kappa_i,\kappa_\rho,\kappa_{p_j}\in\frac12\N$ are the local
exponent data in ($\mb H_1$) and ($\mb H_2$). Hence,
$\alpha_1\in\frac12+\Z$ and $\alpha_2=\pm\frac13+\Z$ and the condition
\eqref{A.4} is equivalent to $r_\infty\neq1/12,5/12$. Thus, the first
half of Theorem \ref{theorem: 1.1} is a special case of Eremenko and
Tarasov's theorem. In the remainder of the appendix, we provide an
alternative and self-contained proof of Theorem \ref{theorem: ET}.

For $k=1,2,3$, let $\theta_k\in(0,1/2]$ be real numbers such
that 
\begin{equation}\label{eq5}
\alpha_k\equiv\pm\theta_k\mod 1,\quad\text{and}\quad\alpha_k=\ell_k\pm\theta_k.
\end{equation}
Let $S=\{0,1,\infty,p_1,\ldots,p_m\}$ be the set of singular points of
\eqref{A.3}. Choose a base point $z_0$ near $\infty$ and consider the monodromy
represenation $\rho: \pi_1(\mathbb{C}\setminus S,z_0)\to
\SL(2,\mathbb{C})$ of \eqref{A.3}.  Let $\beta_j,
\gamma_k\in\pi_1(\mathbb{C}\setminus S,z_0)$ such that $\beta_j$,
$1\leq j\leq m$, (resp. $\gamma_0$, $\gamma_1$) is a simple loop
encircling $p_j$ (resp. $0$, $1$) counterclockwise, while
$\gamma_\infty$ is a simple loop around $\infty$ clockwise such that
\[
  \gamma_0\gamma_1\prod_{j=1}^m \beta_j=\gamma_{\infty},\quad
  \text{in }\pi_1(\mathbb{C}\setminus S,z_0).
\]
Since the local exponents at $\infty$ are
$\{-\alpha_3/2,1+\alpha_3/2\}$ with
$\alpha_3=\ell_3\pm \theta_3$ and any solution has no
logarithmic singularities, we can choose local solutions
$y_{\infty,+}$, $y_{\infty,-}$ near $\infty$ such that with respect to
$(y_{\infty,+},y_{\infty,-})$, the monodromy matrix
$\rho(\gamma_{\infty})$ is given by
\begin{equation} \label{A.T}
\begin{split}
  \rho(\gamma_\infty)
  &=\M{e^{\pi i(\theta_3\pm \ell_3)}}{0}{0}{e^{-\pi i(\theta_3\pm \ell_3)}}\\
  &=(-1)^{\ell_3}
  \M{e^{\pi i\theta_3}}{0}{0}{e^{-\pi i\theta_3}}=:(-1)^{\ell_3}T.
\end{split}
\end{equation}
For any $1\leq j\leq m$, since the local exponents at $p_j$ are
$\{-n_j/2,1+n_j/2\}$ with $n_j\in\mathbb{N}$, we see that the
monodromy matrix $\rho(\beta_j)$ is $(-1)^{n_j} I_2$. Set
\begin{equation} \label{A.RS}
  R:=(-1)^{\ell_1}\rho(\gamma_0)^{-1}, \qquad
  S:=(-1)^{\ell_2}\rho(\gamma_1).
\end{equation}
We have
\[
  (-1)^{\ell_1+\ell_2}R^{-1}S\prod_j^m (-1)^{n_j}
  I_2=(-1)^{\ell_3}T,
\]
i.e.,
\begin{equation}\label{eq8}
  S=(-1)^{\sum_j n_j +\sum_k\ell_k}RT.
\end{equation}

Let $R$, $S$, and $T$ be three matrices in $\SL(2,\C)$ such that
\begin{enumerate}
\item[(i)] the eigenvalues of $R$, $S$, and $T$ are $\delta_1^{\pm 1}$,
  $\delta_2^{\pm 1}$ and $\delta_3^{\pm 1}$, respectively, where
  $\delta_j=e^{\pm\pi i\theta_j}$ with $0<\theta_j<1$ and $i=\sqrt{-1}$,
\item[(ii)] the triple $(\theta_1,\theta_2,\theta_3)$ satisfies
$$
0<\theta_i+\theta_j\leq 1,\quad \forall i\neq j,
$$
and
\item[(iii)] $\theta_3=\max_{1\leq j\leq 3}\theta_j$ and
  $T=\operatorname{diag}(\delta_3,\bar{\delta}_3)=\SM{\delta_3}00
  {\bar{\delta}_3}\in\SU(2,\C)$. 
\end{enumerate}

\begin{Lemma}\label{thm2-1}
Suppose $R=\SM abcd$, $T$, $S=RT\in\SL(2,\mathbb{C})$ satisfy
(i)-(iii). Then the following hold.
\begin{enumerate}
\item[(a)]
$\abs a<1$ if and only if $\theta_1+\theta_2>\theta_3$.
\item[(b)] $\abs a=1$ if and only if $\theta_1+\theta_2=\theta_3$.
\end{enumerate}
\end{Lemma}

\begin{proof}
Note
$$
S=RT=\M abcd\M{\delta_3}00{\bar{\delta}_3}
=\M{\delta_3a}{b\bar{\delta}_3}{\delta_3 c}{d\bar{\delta}_3}.
$$
Using the invariance of $\tr R$ and $\tr S$ under conjugation, we have
$$
\begin{cases}
a+d=\delta_1+\bar{\delta}_1\in\mathbb{R},\\
\delta_3a+\bar{\delta}_3d=\delta_2+\bar{\delta}_2\in\mathbb{R}.
\end{cases}
$$
Since $\delta_3\neq \pm 1$, we easily obtain
\begin{equation}\label{eq14}
  d=\bar{ a},\quad
  a=\frac{\delta_2+\bar{\delta}_2-
\bar{\delta}_3(\delta_1+\bar{\delta}_1)}{\delta_3-\bar{\delta}_3}.\end{equation}
Consequently,
\begin{align*}
a=\frac{2\cos\pi\theta_2-2\bar{\delta}_3\cos\pi\theta_1}{\pm 2i\sin\pi\theta_3}
=\pm\frac{i(\bar{\delta}_3\cos\pi\theta_1-\cos\pi\theta_2)}{\sin\pi\theta_3}.
\end{align*}Thus
\begin{equation}\label{eq3}
\begin{split}
\abs a^2&=\frac{(\bar{\delta}_3\cos\pi\theta_1-\cos\pi\theta_2)(\delta_3\cos\pi\theta_1-\cos\pi\theta_2)}{\sin^2\pi\theta_3}\\
&=\frac{\cos^2\pi\theta_1-2\cos\pi\theta_1\cos\pi\theta_3\cos\pi\theta_2+\cos^2\pi\theta_2}{\sin^2\pi\theta_3}.
\end{split}
\end{equation}
Let
\begin{align*}
  \Delta&:=\cos^2\pi\theta_1-2\cos\pi\theta_1\cos\pi\theta_2\cos\pi\theta_3
          +\cos^2\pi\theta_2-\sin^2\pi\theta_3\\
&=\cos^2\pi\theta_1+\cos^2\pi\theta_2+\cos^2\pi\theta_3-(1+2\cos\pi\theta_1\cos\pi\theta_2\cos\pi\theta_3).
\end{align*}
Then \eqref{eq3} implies that $\Delta<0$ if and only if $\abs a<1$.

Now using the formulas $\cos(x+y)=\cos x\cos y-\sin x\sin y$
and $\cos^2x=(1+\cos(2x))/2$,
we deduce that
\begin{equation*}
  \begin{split}
    \Delta&=\cos^2\pi\theta_3-\cos\pi\theta_3
    (\cos\pi(\theta_1+\theta_2)+\cos\pi(\theta_1-\theta_2)) \\
    &\qquad+\frac12(\cos(2\pi\theta_1)+\cos(2\pi\theta_2)) \\
    &=\cos^2\pi\theta_3-\cos\pi\theta_3
    (\cos\pi(\theta_1+\theta_2)+\cos\pi(\theta_1-\theta_2)) \\
    &\qquad+\cos\pi(\theta_1+\theta_2)\cos\pi(\theta_1-\theta_2),
  \end{split}
\end{equation*}
so
\[\Delta=\left(\cos\pi\theta_3-\cos\pi(\theta_1+\theta_2)\right)
  \left(\cos\pi\theta_3-\cos\pi(\theta_1-\theta_2)\right).\]
Since the assumptions (i)-(iii) give $1>\theta_3>|\theta_1-\theta_2|$, we have $\cos\pi\theta_3-\cos\pi(\theta_1-\theta_2)<0$, so the desired results follow. The proof is complete.
\end{proof}

We now give an alternative proof of Theorem \ref{theorem: ET}, which
is stated in the following equivalent form.

\begin{Theorem}\label{thm1}
  Assume that \eqref{A.4} holds.
\begin{enumerate}
\item[(a)] Suppose that $\sum_{k=1}^3\ell_k+\sum^m_{j=1}n_j$ is an even
  integer. Then \eqref{eq1} has a solution if and only if
  $\theta_i+\theta_j>\theta_k$ for any $i\neq j\neq k$.
\item[(b)] Suppose that $\sum_{k=1}^3\ell_k+\sum^m_{j=1}n_j$ is an odd
  integer. Then \eqref{eq1} has a solution if and only if
   $\theta_1+\theta_2+\theta_3>1$.
\end{enumerate}
\end{Theorem}

\begin{proof} Let $R$, $S$, and $T$ be defined by \eqref{A.T} and
  \eqref{A.RS}. We need to determine when they are simultaneously
  conjugate to unitary matrices, under the assumption that \eqref{A.4}
  holds.

  Consider first the case $\sum\ell_k+\sum n_j$ is even. In such a
  case, we have $S=RT$. Since for any permutation $\tau$ of the three
  points $0$, $1$, and $\infty$, there is always a M\"obius
  transformation $\gamma$ satisfying $\gamma z=\tau(z)$ for all
  $z\in\{0,1,\infty\}$, without loss of generality, we may assume that
  $\theta_3=\max_k\theta_k$. Then the condition
  $\theta_i+\theta_j>\theta_k$ for any $i\neq j\neq k$ simply means
  $\theta_1+\theta_2>\theta_3$, which we assume now. Moreover, we may assume that
  $T=\SM{\delta_3}00{\overline\delta_3}$ after a common conjugation,
  where $\delta_3=e^{\pi i\theta_3}$.

  Write $R=\SM abcd$. By \eqref{eq14}, we have $d=\overline a$.
  By Lemma \ref{thm2-1}, $\theta_1+\theta_2>\theta_3$ if and only if
  $|a|<1$ and hence $bc=|a|^2-1<0$. Set $P=\SM\mu001$, where $\mu$ is
  a real number such that
  $$
  \mu^2=-\frac{bc}{|c|^2}=-\frac b{\overline c}.
  $$
  We have $P^{-1}TP=T$ and
  $$
  P^{-1}RP=\M a{\mu^{-1}b}{\mu c}d,
  $$
  which is unitary since $\mu^{-1}b=-\mu\overline c=-\overline{\mu
    c}$. This proves that if $\theta_1+\theta_2>\theta_3$, then
  \eqref{eq1} has a solution.

  Conversely, suppose that \eqref{eq1} has a solution. Then there
  exists a matrix $P$ such that $\hat T=P^{-1}TP$ and $\hat R=P^{-1}RP$ are both
  unitary. Now it is known that every matrix in $\SU(2,\C)$
  is conjugate to a diagonal matrix and the conjugation can be taken
  inside $\SU(2,\C)$. Hence, there exists a matrix $Q$ in $\SU(2,\C)$
  such that $Q^{-1}\hat TQ=T$. Then $Q^{-1}\hat RQ\in\SU(2,\C)$. In
  particular, the $(1,1)$-entry of 
  $Q^{-1}\hat RQ$ has absolute value $\le 1$. Since $PQ$ commutes with
  $T$ and $T$ is diagonal but not a scalar matrix, $PQ$ must be a
  diagonal matrix. Therefore, 
  the $(1,1)$-entry of $R$ also has absolute value $\le 1$. It follows
  that, by Lemma \ref{thm2-1}, $\theta_1+\theta_2>\theta_3$ (as
  the case $\theta_1+\theta_2=\theta_3$ is excluded from our
  consideration by \eqref{A.4}). We conclude that under the assumptions that
  \eqref{A.4} holds and that $\sum\ell_k+\sum n_j$ is even,
  \eqref{eq1} has a solution if and only if
  $\theta_i+\theta_j>\theta_k$ for any $i\neq j\neq k$.

  For the case $\sum\ell_k+\sum n_j$ is odd, we simply
  apply the result in Part (a) to $\theta_1,\theta_2,1-\theta_3$ with
  $T$ replaced by $-T$ and
  conclude that \eqref{eq1} has a solution if and only if
  $\theta_1+\theta_2+\theta_3>1$. This completes the proof.
\end{proof}

\end{document}